\documentclass{amsart}

\usepackage{amsmath}
\usepackage{amssymb}
\usepackage{mathtools}
\usepackage{mathrsfs}
\usepackage{comment}

\usepackage[english]{babel}
\usepackage[utf8]{inputenc}
\usepackage[T1]{fontenc}

\usepackage[top=3cm,bottom=2cm,left=3cm,right=3cm,marginparwidth=2cm]{geometry}

\usepackage{graphicx}
\usepackage[colorinlistoftodos]{todonotes}
\usepackage[colorlinks=true, allcolors=blue]{hyperref}
\usepackage[noadjust]{cite}
\usepackage{enumitem}
\setlist[1]{itemsep=3pt}
\usepackage{tikz}
\usetikzlibrary{calc,arrows.meta,cd}
\usepackage[capitalise]{cleveref}
\usepackage{setspace}
\usepackage{subcaption}

\expandafter\def\csname ver@etex.sty\endcsname{3000/12/31}

\usepackage{autonum}
\usepackage{thm-restate}

\newcommand{\R}{\mathbb{R}}

\newcommand{\dd}{\mathrm{dist}}

\newcommand{\N}{\mathbb{N}}

\newcommand{\dist}{{\mathrm{dist}_Y|_X}}

\newcommand{\Poly}{\mathscr{P}}

\def\Id{\operatorname{Id}}
\def\lini{k}
\def\quadi{\iota}

\newtheorem{theorem}[subsubsection]{Theorem}
\newtheorem{proposition}[subsubsection]{Proposition}
\newtheorem{corollary}[subsubsection]{Corollary}
\newtheorem{lemma}[subsubsection]{Lemma}

\theoremstyle{definition}
\newtheorem{definition}[subsubsection]{Definition}

\theoremstyle{remark} 
\newtheorem{remark}[subsubsection]{Remark}
\newtheorem{example}[subsubsection]{Example}

\newcommand{\be}{\begin{equation}}
\newcommand{\ee}{\end{equation}}

\title[Morse Theory of Euclidean Distance Functions from Algebraic Hypersurfaces]{Morse Theory of Euclidean Distance Functions from Algebraic Hypersurfaces}
\author{Andrea Guidolin, Antonio Lerario, Isaac Ren, and Martina Scolamiero}
\address{University of Southampton, School of Mathematical Sciences, SO17 1BJ Southampton, UK}
\email{Andrea Guidolin <a.guidolin@soton.ac.uk>}
\address{KTH, Department of Mathematics, S-10044 Stockholm, Sweden}
\email{Isaac Ren <isaacren@kth.se>}
\email{Martina Scolamiero <scola@kth.se>}
\address{SISSA, Mathematics Area, 34146 Trieste, Italy}
\email{Antonio Lerario <lerario@sissa.it>}
\date{\today}
\thanks{Andrea Guidolin was supported by the dBrain collaborative project at Digital Futures at KTH,  by the Strategic Support grant of the Digitalisation Platform at KTH and by the Data Driven Life Science (DDLS) program funded by the Knut and Alice Wallenberg Foundation. Antonio Lerario was supported by the Knut and Alice Wallenberg Foundation. Isaac Ren was supported by Vetenskapsrådet. Martina Scolamiero was supported by the Wallenberg AI, Autonomous System and Software Program (WASP) funded by the Knut and Alice Wallenberg Foundation. Martina Scolamiero is affiliated with the research center Digital Futures, Stockholm, Sweden}

\begin{document}

\begin{abstract}
Let $Y\subseteq \R^n$ be a closed definable subset and $X\subseteq \R^n$ be a smooth manifold. We construct a version of Morse theory for the restriction to $X$ of the Euclidean distance function from $Y$. This is done using the notion of critical points of Lipschitz functions and applying the theory of continuous selections. In this theory, nondegenerate critical points have two indices: a quadratic index (as in classical Morse theory), and a piecewise linear index (that relates to the notion of bottlenecks).
This framework is flexible enough to simultaneously treat and unify the study of two cases of interest for computational algebraic geometry: bottlenecks and nearest point problems.
We provide a technical toolset guaranteeing the applicability of the theory to the case where $X, Y$ are generic algebraic hypersurfaces and use it to bound the number of critical points of the distance from $Y$ restricted to $X$, among other applications.
\end{abstract}

\maketitle

\section{Introduction}

\subsection{A Morse theory for distance functions}
Distance functions have been the object of extensive study over the years, classically in differential geometry \cite{Looijenga1974,Wall1977,Grove_annals,Yomdin1981,cheeger,BrockerKupper2000,Rifford2004,CohenSteinerCommaret2024}, and more recently in computational algebraic geometry  \cite{Ottaviani2018TheDF,DiRocco_Eklund_Weinstein,DEEGH2023,BSK} and computational topology \cite{Chazal_Lieutier,bobrowski2014distance,oudot2017persistence,song2023generalized,ACSD2023}.
Various interesting properties of geometric objects in $\R^n$ (e.g.\ the Euclidean Distance Degree, bottlenecks) are defined in terms of Euclidean distance and arise as some type of ``critical points'', meaning that they are defined as set of solutions of algebraic systems of equations resembling the equations for critical points of a smooth function.
Our goal here is to put these ideas  into a general systematic framework, explaining how to recover all these notions at once using the language of Lipschitz geometry, and giving  a topological interpretation of them. 

Studying critical points of functions has many applications ranging from data analysis tasks such as the characterization of the texture of porous materials in \cite{song2023generalized} to machine learning where for example critical points of objective functions over Linear Convolutional Neural Networks are studied to understand their optimization \cite{Kohn_22, Kohn_24}. Low rank approximation problems for tensors as the one in Example \ref{lowrank_approx:example} are also relevant in machine learning as well as in signal processing and data compression \cite{Bader_tensors,MAL-059}. Furthermore bottlenecks can be used to produce provably dense samples of algebraic varieties, capturing the homology of data spaces naturally encoded by algebraic equations \cite{Sampling_bottlenecks}.

Given two real algebraic subsets $X, Y \subseteq \R^n$, in this paper we will explain how to construct a version of Morse theory for the function
\[
\dist \colon X \to \R,
\]
where $\dd_Y(z) \coloneqq \min_{y\in Y} \| y - z \|$ denotes the distance function from $Y$ and $\dist$ denotes its restriction to $X$. For the applications we have in mind, both $X$ and $Y$ will be smooth algebraic hypersurfaces, but some of the results we will formulate are valid under weaker assumptions (e.g.\ in the definable\footnote{In this paper, ``definable'' means definable in some o-minimal structure, for instance in the structure of globally subanalytic sets.} world or in the smooth world under some transversality conditions).

Compared to the classical theory, the key point here is that, even under the assumption that $X$ and $Y$ are smooth and algebraic, in general the function $\dd_Y$ is not differentiable, and similarly neither is $\dist$. What is then the definition of a critical point for functions of this type? If $X$ and $Y$ are semialgebraic, it is true that $\dist$ is semialgebraic, and  one can at least set up a notion of ``critical value'' for it (see \cref{sec:triviality}), but this seems unsatisfactory: in particular, little of the geometry of distance functions is used.  Instead, we base our approach on the notion of critical points of Lipschitz functions introduced by Clarke \cite{Clarke} and apply the more general Morse theory of continuous selections, as discussed in  \cite{APS1997}. The applicability of this theory to the context of our interest requires some nondegeneracy conditions to be verified (as is the case for classical Morse theory!) and the technical part of the paper proves that these conditions are generically verified in the case of algebraic hypersurfaces (see \cref{S:Genericity-algebraic}).

Our framework is flexible enough to include the two cases of interest for computational algebraic geometry: when $X$ is $\R^n$ and $Y$ is a smooth algebraic set, the critical points of $\dist$ (as defined below) leads to the notion of Bottleneck Degree \cite{DiRocco_Eklund_Weinstein}; when $Y$ reduces to a point and $X$ is an algebraic set, they  lead to the notion of Euclidean Distance Degree \cite{EDD_Draisma}, see \cref{sec:related} below.

In our exposition we focus on the case when either $X$ or $Y$ (or both) are smooth real algebraic \emph{hypersurfaces}, and genericity is considered with respect to the defining polynomials. We also treat the case when $X$ is $\R^n$ and $Y$ is a finite set of points, which is relevant for applications.

\begin{example}\label{lowrank_approx:example} Given a point cloud in the space of tensors, what is the closest rank-one tensor to this point cloud? This problem can be formulated as finding the critical point with smallest value of $\dist$, where $Y=\{y_0, \ldots, y_m\}$ is the point cloud in the space of tensors $\R^{n_1}\otimes\cdots \otimes\R^{n_j}$ and $X$ is the set of rank-one tensors.
\end{example}
We expect that most of the results extend to the case of complete intersections, but we leave this to future work. Instead, we provide a technical toolset (most notably a multijet parametric transversality theorem for polynomials, see \cref{sec:multijet}) which the reader will easily be able to apply in their case of interest.

\subsection{Critical points of distance functions}
Let us first explain how to define ``critical points'' in our context. The key observation is that the function $\dist$ is Lipschitz. Following Clarke \cite{Clarke}, for such functions, one can define the notion of \emph{subdifferential}. This is done for a general locally Lipschitz function $f \colon X \to \R$ as follows (see \cref{def:subdifferential}). 

Denoting by $\Omega(f) \subseteq X$ the set of points where $f$ is differentiable (of full measure by Rademacher's Theorem), the \emph{subdifferential of $f$ at $x\in X$}, denoted by $\partial_x f$, is the convex body
\[
\partial_xf \coloneqq \mathrm{co}\left\{ \left. \lim_{x_k \to x, x_k \in \Omega(f)} D_{x_k} f \; \right| \; \textnormal{the limit exists}\right\} \subseteq (T_x X)^*.
\]
It is often practical to assume that $X\subseteq \R^n$ is a Riemannian manifold with the metric induced by the Euclidean metric of $\R^n$, so that one can identify $T_x X \cong (T_x X)^*$ and view the subdifferential as a convex body in $T_x X$.

\begin{definition}[Critical points, \cref{D:critical}]
A \emph{critical point} of $f$ is a point $x\in X$ such that the subdifferential $\partial_xf$ contains the zero vector.
\end{definition}

\begin{figure}
\centering

\begin{tikzpicture}
\node (x) at (3, 2) {};
\node (y1) at (1, 2) {};
\node (y2) at (5, 2) {};
\coordinate (A) at (1.5, -1);
\coordinate (B) at (5, 6);
\coordinate (p1) at ($(A)!(y1)!(B)$);
\coordinate (p2) at ($(A)!(y2)!(B)$);

\coordinate (C) at ($(x)!4cm!90:(A)$);

\draw[line width=2pt] (-1, 0) .. controls (0, 0) and (1, 1) .. (1, 2)
  .. controls (1, 3) and (0.5, 3) .. (0.5, 4)
  .. controls (0.5, 5) and (1.5, 6) .. (3, 6)
  .. controls (4.5, 6) and (5.5, 5) .. (5.5, 4)
  .. controls (5.5, 3) and (5, 3) .. (5, 2)
  .. controls (5, 1) and (6, 0) .. (7, 0) node[right]{$Y$};
\draw (x) circle (2);
\draw[line width=3ex, color=white!80!black] (p1) -- (p2);
\draw (A) -- (B) node[right]{$T_x X$};
\draw[line width=2pt] ([shift=(90:4cm)]C) node[right]{$X$} arc (90:195:4cm);
\draw[->, line width=1pt] (x) -- node[above]{$\frac{x - y_1}{\| x - y_1 \|}$} (y1);
\draw[->, line width=1pt] (x) -- node[below]{$\frac{x - y_0}{\| x - y_0\|}$} (y2);
\draw[dashed] (y1) -- (p1)
  (y2) -- (p2);
\node[above left] at (p2) {$\partial_x f$};

\filldraw (x) circle (2pt) node[above left]{$x$};
\filldraw (y1) circle (2pt) node[left]{$y_0$};
\filldraw (y2) circle (2pt) node[right]{$y_1$};
\end{tikzpicture}

\caption{Example of the subdifferential of $f = \dist$ at a critical point $x$. The subdifferential $\partial_x f$ is the interval between the projections of the vectors $\frac{x - y_0}{\| x - y_0 \|}$ and $\frac{x - y_1}{\| x - y_1 \|}$ onto $T_x X$.}
\label{F:distance-subdifferentialintro}
\end{figure}

The set of critical points of $f$ is denoted by $C(f)$. A regular point is a point that is not critical.  
The definitions of critical value and regular value are the usual ones. Minima and maxima are critical points; moreover, at a point of differentiability, one recovers the classical definitions. Going back to our functions of interest, if $Y \subseteq \R^n$ is a closed set and $x \in \R^n \setminus Y$, then the subdifferential of $\dd_Y$ at $x$ can be easily described, see \cref{propo:subgradient1}:
\[\label{eq:subintro}
\partial_x(\dd_Y) = \mathrm{co}\left\{
\frac{x - y}{\|x - y\|} \,\middle|\, y \in B(x, \dd_Y(x)) \cap Y \right\},
\]
where $B(x, \dd_Y(x))$ is the closed ball of radius $\dd_Y(x)$ centered in $x$.
If $X$ and $Y$ are in general position (see \cref{propo:subgradient2}) The subdifferential of $\dist$ at $x\in X$ is just the projection of $\partial_x(\dd_Y) $ on the tangent space $T_xX$, see Figure \ref{F:distance-subdifferentialintro}.

Various statements from classical analysis extend to the Lipschitz framework of distance functions. For instance, under the only assumption that $Y$ is definable, the set of critical values of $\dist$ is definable\footnote{Definability is needed, see \cref{remark:needed}.} and of measure zero (this is called the ``Sard property'', see \cref{sec:sard}). Another example of a classical result that extends to this framework is the Deformation Lemma, whose proof is based on Clarke's Implicit Function Theorem \cite{Clarke}. In order to state the lemma, we define for all $t \in \R$ the sublevel set
\[
X_Y^t \coloneqq \left\{ x\in X \,\bigg|\, \dd_Y(x)\leq t \right\}.
\]

\begin{proposition}[Deformation Lemma, \cref{lem:deformation}]
If the interval $[a,b] \subset \R$ contains no critical values of $\dist$, then the set $X_Y^b$ deformation retracts to the set $X_Y^a$.
\end{proposition}

\subsection{Nondegenerate critical points} In the definable framework one can even prove that the cohomology of $X_Y^t$ changes in a controlled manner when passing a critical value (\cref{propo:criticalpass1}) but, as in the classical theory, in order to be able to say more, one needs to introduce the notion of \emph{nondegenerate} critical points. This is done using the theory of \emph{continuous selections} as follows.

Following \cite{APS1997}, we say that $f \colon X \to \R$ is a continuous selection of the finite list $(f_0, \ldots, f_m)$ of $\mathcal{C}^2$ functions  if $f$ is continuous and if for every $x \in X$ there exists $i \in \{0, \ldots, m\}$ such that $f(x) = f_i(x)$. In this case we write $f \in \mathrm{CS}(f_0, \ldots, f_m)$. For $f \in \mathrm{CS}(f_0, \ldots, f_m)$ and $i \in \{0, \ldots, m\}$, we denote by $S_i \subseteq X$ the set
\[
S_i \coloneqq \left\{ x\in X \,\bigg|\, f(x) = f_i(x) \right\}.
\]
Given $x \in X$, the \emph{effective index set} of $f\in \mathrm{CS}(f_0, \ldots, f_m)$ at $x\in X$ is defined as
\[
I(x)\coloneqq \left\{ i\in \{0, \ldots, m\} \,\bigg|\, x \in
\overline{\operatorname{int}(S_i)}
\right\}.
\]
Note that continuous selections are locally Lipschitz, so we can talk about their critical points using the language introduced above. We can think of these functions as ``piecewise $\mathcal{C}^2$ functions''.

Now comes the crucial definition. 

\begin{definition}[Nondegenerate critical points, \cref{D:nondegenerate-crit-pt}]
\label{def:nondegintro}
For $f \in \mathrm{CS}(f_0, \ldots, f_m)$, a critical point $x \in C(f)$ is said to be \emph{nondegenerate} if the following conditions are satisfied:
\begin{enumerate}[label=(\roman*)]
\item \label{eq:indexset} For every $i\in I(x)$ the set of differentials $\{D_x f_j \mid j\in I(x)\setminus \{i\}\}$ is linearly independent.
\item Denote by $(\lambda_i)_{i\in I(x)}$ the unique list of  numbers such that 
\[
\sum_{i \in I(x)} \lambda_i = 1
\quad \textrm{and} \quad
\sum_{i \in I(x)} \lambda_i D_x f_i = 0
\]
and by $\lambda f \colon X\to \R$ the function defined by $\lambda f(z) \coloneqq \sum_{i\in I(x)} \lambda_i f_i(z).$ Then the second differential of $\lambda f$ is nondegenerate on 
\[
V(x) \coloneqq \bigcap_{i\in I(x)} \ker(D_x f_i).
\]
\end{enumerate}
\end{definition}

When $f\in \mathrm{CS}(f_0)$, $f$ is $\mathcal{C}^2$ and the above definition reduces to the usual definition of nondegenerate critical point, in the sense of Morse theory. 

Compared to what happens with classical Morse functions, we see that a nondegenerate critical point of a continuous selection in general comes with two indices: \begin{enumerate}
\item [-]the number  $\lini(x) \coloneqq \#I(x) - 1$, which we call the \emph{piecewise linear index};
\item [-]the number $\quadi(x)$ of negative eigenvalues of the restriction of the second differential of $\lambda f$ to $V(x)$, which we call the \emph{quadratic index}.
\end{enumerate} 
Continuous selections have normal forms near their nondegenerate critical points, in a fashion similar to the Morse Lemma \cite[Lemma~2.2]{milnor1963morse}, see \cref{rmk:normalformdist1}. Therefore, as for the classical Morse theory, for a continuous selection $f$ the cohomology of the sublevel sets $\{f \leq t\}$  changes in a controlled way when passing a critical value with only nondegenerate critical points (see \cite{APS1997} and \cref{propo:passmin}), and this change can be bounded by the two numbers $\lini$ and $\quadi$ only. For instance, in our context, we have the following result.

\begin{proposition}[\cref{propo:criticalpass2dist}]
\label{propo:passintro}
If $\dist$ is a continuous selection at $x \in X$ and $x$ is a nondegenerate critical point for it such that $C(\dist)\cap (\dist)^{-1}(c) = \{x\}$, then for all $\epsilon > 0$ small enough,
\[
H^*(X^{c + \varepsilon}_Y, X^{c - \varepsilon}_Y) \cong \widetilde{H}^*(S^{\lini(x)+\quadi(x)}).
\]
\end{proposition}

\subsection{Critical points and bottlenecks}
In the statement of \cref{propo:passintro}, the assumption that $\dist$ is a continuous selection is essential. Indeed, in general, distance functions might not fall in this class, even under seemingly nice assumptions, see \cref{example:nonCS}. We show, however, that this is generically the case in \cref{C:C2selection}.

Note that, due to the description \eqref{eq:subintro}, in order to satisfy the condition \ref{eq:indexset} from \cref{def:nondegintro} we need in particular that, at each critical point $x \in C(\dist)$, the function $m_Y \colon \R^n\to \R$, defined by
\[
m_Y(x) \coloneqq \# \Big( B(x, \dd_Y(x))\cap Y \Big),
\]
should be bounded by $n+1$.
In this direction, Yomdin proved in \cite[Proposition 3]{Yomdin1981} that this is true for a generic embedding $i \colon Y \to \R^n$, when $Y$ is a compact smooth manifold. Here we prove the following result, which is an analogue of Yomdin's result in the algebraic framework (but does not follow from it).
For $d\in \mathbb{N}$, we denote by $\Poly_d$ the set of polynomials on $\R^n$ of degree at most $d$.

\begin{theorem}[\cref{T:Yomdin}]
\label{T:Yomdinintro}
For generic $q \in \Poly_d$ with $d \geq 2$, writing $Y = Z(q)$,  and for every $x\in\R^n$, the set $B(x, \dd_Y(x)) \cap Y$ is a nondegenerate simplex with at most $n + 1$ points. In particular, $m_Y \leq n + 1.$
\end{theorem}

Now let $Y \subseteq \R^n$ be an algebraic set. Following \cite{DEEGH2023}, one says that a critical point $x \in C(\dd_Y)$ is a \emph{geometric $(k+1)$-bottleneck}, also called bottleneck of order $k$, if 
$x$ is a convex combination of $k+1$ affinely independent points in $B(x, \dd_Y(x))\cap Y$ and is not a convex combination of a smaller number of points in $B(x, \dd_Y(x))\cap Y$. In particular, a geometric $(k+1)$-bottleneck $x$ satisfies $m_Y(x) \ge k+1$.
Under the assumption that $\dd_Y$ is a continuous selection with only nondegenerate critical points (this is generically true, see \cref{T:genericity-nondegeneracy}), geometric $(k+1)$-bottlenecks are precisely critical points of $\dd_Y$ with piecewise linear index equal to $k$ (see \cref{def:nondegintro}) and an analogue of the cellular decomposition theorem is \cref{propo:criticalpass2dist}. 

Motivated by the connection between the order of geometric bottlenecks and the piecewise linear index,
we introduce the following general definition, which, in the nondegenerate case, partitions the critical points according to their \emph{piecewise linear} index. 

\begin{definition}[$k$-critical points, \cref{D:k-crit-pt}]
\label{def:kcrit}
Let $X\subseteq \R^n$ be a smooth manifold and $Y\subseteq \R^n$ be a closed definable set. For all $k \geq 0$, we define the \emph{$k$-critical points of $\dist$} as the critical points $x \in X$ of $\dist$ such that  $B(x, \dd_Y(x)) \cap Y=\{y_0, \ldots, y_k\}$.
We denote by $C_{k, *}(\dist)$ the set of $k$-critical points\footnote{At these critical points the quadratic index might not be defined, as they are not necessarily nondegenerate.} of $\dist$ not in $Y$. 

Note that if $X\cap Y\neq \emptyset$, then every point in this intersection is critical for $\dist$, in fact these are precisely the global minima.

In the nondegenerate case (so that the quadratic and piecewise linear indices are defined), $k$-critical points have piecewise linear index $k$. We denote by $C_{\lini, \quadi}(\dist) \subseteq C_{\lini, *}(\dist)$ the set of nondegenerate critical points of $\dist$ with piecewise linear index $\lini$ and quadratic index $\quadi$.
For instance, a geometric $(k+1)$-bottleneck that is nondegenerate with quadratic index $\quadi$ is contained in $C_{\lini, \quadi}(\dd_Y)$.
\end{definition}

In the generic algebraic case, by \cref{T:Yomdinintro}, every critical point is a $k$-critical point for some $k \in \{0, \ldots, n\}$.
In fact, a stronger result is true. The next theorem collects various statements from \cref{S:Genericity-algebraic} and shows that, in the generic case, all of the ideas that we have introduced so far combine to give a Morse theory for the function $\dist$ in the case $X,Y$ are smooth algebraic hypersurfaces.

\begin{theorem}
\label{T:finite-crit-ptintro}
Let $p \in \Poly_{d_1}$ and $q\in \Poly_{d_2}$ be generic polynomials, and let $X = Z(p)$ and $Y = Z(q)$ be their zero sets. Then, depending on the values of $d_1=\mathrm{deg}(X)$ and $d_2=\mathrm{deg}(Y)$ the following is true:
\begin{itemize}
\item[-] if $d_1 \geq 3$ and $d_2 \geq 4$, then the function $\dist$ is a continuous selection near each of its critical points (\cref{C:C2selection});
\item[-] if $d_1, d_2 \geq 4$, then every critical point of $\dist$ not in $Y$ is nondegenerate (\cref{T:genericity-nondegeneracy}); in particular \cref{propo:passintro} can be applied;
\item[-] if $d_1, d_2 \geq 3$, then for every $k\in \{1, \ldots, n\}$, the number of $\lini$-critical points of $\dist$ not in $Y$ can be bounded by
\be\label{eq:boundcriticalintro}\#C_{\lini, *}(\dist)\leq c(k,n)\mathrm{deg}(X)^n\mathrm{deg}(Y)^{n(k+1)},\ee
for some $c(k,n)>0$ depending on $n$ and $k$ only (\cref{T:finite-crit-pt}).
\end{itemize}
\end{theorem}

\begin{remark}
In the case of $(X, Y) = (\R^n, Y)$, which is not covered by the above theorem,  the result still applies, see \cref{R:X-is-Rn-case}. The fact that, for every $k \in \{1, \ldots, n\}$, there are generically finitely many $k$-critical points not in $Y$ follows from \cite[Theorem 4.14]{DEEGH2023} (which holds for non-quadratic complete intersections). However, \cite[Theorem 4.14]{DEEGH2023} does not prove that this number is zero for $k\ge n+1$ (i.e.\ it doesn't rule out the possibility of critical points with $m_Y$ arbitrarily large), which is instead a consequence of \cref{T:Yomdinintro}. 
\end{remark}

\begin{remark}[A duality formula]
The partition of nondegenerate critical points according to the \emph{sum} of their piecewise linear and quadratic index provides an interesting duality formula (see \cref{coro:duality}), that for the generic pair $(X, Y)$, if $X$ and $Y$ are compact, reads as
\[
\chi(Y)+\sum_{\lini, \quadi\geq 0}(-1)^{\lini + \quadi}\#C_{\lini, \quadi}(X, Y)=\chi(X)+\sum_{\lini, \quadi\geq 0}(-1)^{\lini + \quadi}\#C_{\lini, \quadi}(Y, X).
\]
For example, when $Y=\{y\}$ with $y \notin X$, this formula gives back the Euler Characteristic of $X$ as the alternating sum of the number of critical points of quadratic index $\quadi$.
\end{remark}

\subsection{Related works on Morse theory of distance functions}\label{sec:related}
Prior to our work there have been several contributions to build a framework for the Morse theory of distance functions: most notably \cite{cheeger}, that deals with distance-from-a-point functions on Riemannian manifolds, and \cite{bobrowski2014distance}, which treats the case of point clouds. 

We note that, in \cite{bobrowski2014distance}, the piecewise linear index $k$ is identified as an analogue of the classical Morse index in the case of distance functions from a finite set of points $Y = \{y_0, \ldots, y_m\}$, and in \cite[Remark 4.2]{DEEGH2023} it is compared to the order of the geometric bottlenecks in the context of more general distance functions (with a shift in index). However, as we explained above, in the general case the piecewise linear index $k$ is not the correct analogue of the classical Morse index. It is true that, for the distance from a finite set of points, the piecewise linear index correctly describes the changes in topology, but this is because the quadratic index in this case is always zero. In general, as demonstrated by \cref{propo:passintro}, changes in topology take into account both the piecewise linear index and the quadratic index.

More recently, the authors of \cite{song2023generalized} developed the Morse theory of signed distance functions from a hypersurface in Euclidean space. Via an argument based on transversality theory, they prove that, for a generic embedding of a smooth compact surface in $\R^3$, the associated signed distance function has finitely many critical points that are all nondegenerate. 
The Morse theory for distance functions we present in Section \ref{S:Morse} is similar to the treatment in \cite[Sect.~3]{song2023generalized}. The main difference between the two approaches is that the authors of \cite{song2023generalized} build on the theory of Min-type functions \cite{gershkovich1997morse} which, like continuous selections, admit normal forms at nondegenerate critical points.
For applications to Euclidean distance functions, the generalized Morse theory built on the theory of Min-type functions is equivalent to that built on continuous selections, as we observe in \cref{rmk:normalformdist}.

\subsection{The Bottleneck Degree, the Euclidean Distance Degree and related notions}
Let us go back to the study of geometric $(k+1)$-bottlenecks in details. As we have seen, they are in particular critical points of $\dd_Y$ with at least $k+1$ closest points in $Y$; in the nondegenerate case, this means critical points with piecewise linear index $\lini$. When $X$ is $\R^n$ and $Y$ is an algebraic hypersurface, the above result can still be applied (see \cref{R:X-is-Rn-case} and Remark $1.4.4$), and the  bound \eqref{eq:boundcriticalintro} becomes
\[
\#C_{\lini, *}(\dd_Y) \leq c(k,n)(\deg(Y))^{n (k + 1)}.
\]
When $k = 1$, geometric $2$-bottlenecks are simply called ``bottlenecks''. These bottlenecks are easier to study, because of the structure of their equations (roughly speaking, they do not involve quantifier elimination). The \emph{Bottleneck Degree} of a smooth variety $Y$ defined over $\R$, denoted by $\mathrm{BND}(Y)$, was defined in \cite{DiRocco_Eklund_Weinstein} as the number of complex solutions of this system of equations for generic $Y$.
The motivation for the study in \cite{DiRocco_Eklund_Weinstein} was to control the number of bottlenecks of (the real part of) a real variety $Y$, and this can be done passing to the complexification and using the bound 
\[\#C_{1,*}(\dd_Y) \leq \mathrm{BND}(Y).\] In \cite{DiRocco_Eklund_Weinstein}, the authors compute various bottleneck degrees (planar and space curves, space surfaces, and  threefolds) and deduce a bound that, for these examples, is of the form $\#C_{1,*}(\dd_Y) \leq \widetilde{c}(1, n) d^{2n}.$ In this sense, our bound \eqref{eq:boundcriticalintro} generalizes this for all dimensions and for bottlenecks of all orders.

A similar discussion holds for the notion of \emph{Euclidean Distance Degree} of a smooth real algebraic variety $X$, denoted by $\mathrm{EDD}(X)$. This was defined in \cite{EDD_Draisma} as the number of complex solutions of the equations describing the critical points of the distance function from a generic point $y \in \R^n$.  

If one is only interested in the real solutions to the critical points equations, using our language, one could define the following notion. For generic pairs $(X, Y)$ of real algebraic sets in $\R^n$, we define the  \emph{Real Euclidean Distance Degree of $X$ relative to $Y$} as the number
\[
\mathrm{EDD}_\R(X, Y) \coloneqq \#C(\dist).
\]
In the special cases $(X, Y) = (\R^n, Y)$ and $(X, Y) = (X, \{y\})$, one could refine this definition and obtain the quantities
\[
\mathrm{BND}_\R(Y) \coloneqq \mathrm{EDD}_\R(\R^n,Y) \quad \textrm{and} \quad \mathrm{EDD}_\R(X) \coloneqq \mathrm{EDD}_\R(X, \{y\}).
\]
In the generic case, as we have seen, both these numbers are actually the number of critical points of an appropriate function (the function $\dist$) and can be bounded by \eqref{eq:boundcriticalintro}. 

Note that, unlike analogue complex notions, this number depends on the pair $(X, Y)$.
It is clear from this discussion, for instance, that $\mathrm{EDD}_\mathbb{R}(X)\leq \mathrm{EDD}(X).$
However, we leave the study of this notion of ``relative EDD'' from the complex point of view to future investigations. The definition of the complex solutions of the system defining our critical points requires some care, because convexity is used, but the constructions from \cref{S:Genericity-algebraic} also work over the complex numbers. 

A natural question for the experts is: after it is defined, how to compute the (complex) Euclidean Distance Degree of $X$ relative to $Y$ using polar classes?

\subsection{Tools for genericity}
Our proof uses \emph{real techniques} as opposed to, for example, the techniques from \cite{DiRocco_Eklund_Weinstein}, which come from complex algebraic geometry and use the notion of polar classes. However, the technical results from the current paper (in particular, those from \cref{sec:multijet}) extend easily to the complex world, with the same proofs. In fact, we propose these techniques as easy-to-use substitute tools to verify that codimensions are the expected ones in the generic case. For example, the proof of finiteness of the number of $(k+1)$-bottlenecks for non-quadratic complete intersections from  \cite[Theorem 4.14]{DEEGH2023} uses the Alexander-Hirschowitz Theorem, that here is substituted with the parametric multijet theorem (\cref{L:submersion}).

\subsection*{Acknowledgements}
The authors would like to thank the anonymous referees whose constructive comments contributed to improve the paper.

\section{Morse theory for distance functions}
\label{S:Morse}

Let $Y\subseteq \R^n$ be a closed definable subset and $X\subseteq \R^n$ a smooth manifold. In this section we will explain how to construct a version of Morse theory for the function 
\[\dist \colon X \to \R,
\]
where $\dd_Y(z) \coloneqq \min_{y\in Y} \| y - z \|$ and $\dist$ denotes the restriction of $\dd_Y$ to $X$. 
Except for some special cases (e.g.\ when $Y=\{y\}$ is a point which is not in $X$), the function $\dist$ is not differentiable and not convex. To formulate basic statements about the change of the topology of its  sublevel sets, one needs to introduce first appropriate notions of critical points and critical values. 
This will be done, first for the function $\dd_Y$ and then for $\dist$, following a  general theory for locally Lipschitz functions. 

\subsection{Critical points}

An appropriate notion of critical points for $\dd_Y$ is introduced as follows. First we observe that $\dd_{Y}$ is Lipschitz, and therefore differentiable almost everywhere by Rademacher's Theorem. Following Clarke \cite{Clarke}, one can therefore define the \emph{subdifferential} of $\dd_{Y}$ at a point $x\in \R^n$. 

\begin{definition}[Subdifferential]\label{def:subdifferential}
Let $(X, g)$ be a Riemannian manifold and  $f \colon X\to \R$
be a locally Lipschitz function. (In the framework we have in mind $X\subseteq \R^n$ is a submanifold with the induced Riemannian metric from the ambient Euclidean space.) We denote by $\Omega(f) \subseteq X$ the set of differentiability points of $f$ (of full measure by Rademacher's Theorem). The \emph{subdifferential of $f$ at $x\in X$}, denoted by $\partial_x f$, is the convex body defined by
\[
\partial_xf \coloneqq \mathrm{co}\left\{ \lim_{x_k \to x, x_k \in \Omega(f)} \nabla f(x_k) \; \middle| \; \textnormal{the limit exists}\right\} \subseteq T_x X.
\]
Note that at a strictly differentiable point, the subdifferential reduces to a singleton $\partial_x f = \{\nabla f(x)\}$.
\end{definition}

\begin{definition}[Critical points]
\label{D:critical}
We say that $x \in X$ is a \emph{critical point} of a locally Lipschitz function $f \colon X\to \R$ if $0$ belongs to $\partial_x f$. We denote by $C(f) \subseteq X$ the set of critical points of $f$. 

A point $x\in X$ is said to be a \emph{regular point} of $f$ if it is not critical; a \emph{critical value} of $f$ is a real number $c \in \R$ such that $f^{-1}(c) \cap C(f) \neq \emptyset$ and a \emph{regular value} of $f$ is a real number $c \in \R$ which is not critical.
\end{definition}

\begin{remark}
For example, local maxima and local minima are critical points of locally Lipschitz functions \cite{Rockafellar}.
\end{remark}

\begin{definition}[Medial axis]\label{def:medial}
Let $Y \subseteq \R^n$ be a closed set. We introduce the function $m_{Y} \colon \R^n \to \R \cup \{\infty\}$, defined for $z \in \R^n$ by
\[
m_Y(z) \coloneqq \#\left(B(z, \dd_Y(z)) \cap Y\right),
\]
where $B(z, \dd_Y(z))$ denotes the closed ball of radius $\dd_Y(x)$ centered in $x$, and the set $M_Y \subset \R^n$, called the \emph{medial axis} of $Y$, defined by
\[
M_{Y} \coloneqq \left\{z\in \R^n\,\bigg|\, m_Y(z)\geq 2\right\}.
\]
\end{definition}

\begin{remark}
\cite[Prop.\ 3]{Yomdin1981} proves that $m_Y$ is finite if $Y = i(W)$ is the image of a  compact smooth manifold $W$ by a generic embedding $i \colon M \to \R^n$. An analogue result in the algebraic setting is shown by \cref{T:Yomdin}. 
\end{remark}

\begin{lemma}
\label{L:useful}
Let $Y \subseteq \R^n$ be a closed definable set. The set $M_Y$ is also definable, of codimension at least one, and the function $\dd_Y$ is continuously differentiable on $\R^n \setminus (Y \cup \overline {M_Y})$.
\end{lemma}

\begin{proof}The fact that $M_Y$ is definable is obvious. The fact that it has codimension at least one follows from the strict convexity of the norm and the fact that $\dd_Y$ is continuously differentiable on $\R^n \setminus (Y \cup \overline {M_Y}))$ is the content of \cite[Theorem 4.8(5)]{Federer1959}.
\end{proof}

\begin{example}[Critical points of the distance function from a finite set]
\label{ex:finiteY}
Let $Y=\{y_0,\ldots ,y_m \}$ be a finite set of points in $\mathbb{R}^n$, and let $X=\mathbb{R}^n$.

The work of Bobrowski and
Adler \cite{bobrowski2014distance} explains the Morse theory for the function  $\dd_Y$; it turns out that this is a special case of the Morse theory from this paper. 

To each $y_i\in Y$ one can associate the \emph{Voronoi cell} $V(y_i) \coloneqq \{ x \in \mathbb{R}^n \mid \| x-y_i \| \leq \| x-y_j \|, \forall y_j \in Y \}$. Voronoi cells are convex polyhedra, and any two different Voronoi cells have disjoint interiors. The collection of Voronoi cells and their faces constitute a polyhedral complex \cite{ziegler2012lectures} called the \emph{Voronoi diagram} of $Y$, see e.g.\ 
\cite{boissonnat2018geometric}, which covers the whole space $\mathbb{R}^n$.

In the interior of a Voronoi cell $V(y_i)$, the distance function $\dd_Y$  coincides with $\dd_{\{y_i\}}$ and its only critical point is $y_i$. The medial axis $M_Y$ is union of the boundaries of all Voronoi cells, and it contains all remaining critical points of $\dd_Y$ (cf.\ Lemma \ref{L:useful}). 
If no subset of $n+2$ points of $Y$ lie on a common hypersphere, then for each $x\in \mathbb{R}^n$ the value $m_Y(x)$ coincides with $n+1$ minus the lowest dimension of a Voronoi face containing $x$.

A set $Y=\{y_0,\ldots ,y_m \}$ in $\mathbb{R}^n$ is in \emph{general position} if every subset of up to $n+1$ points of $Y$ is affinely independent.
If $Y$ is in general position, then one can find the critical points of $\dd_Y$ via the following procedure: for each subset $Y'$ of $\lini+1$ points in $Y$, with $\lini \le n$, let $x$ denote the center of the unique $(\lini-1)$-sphere containing $Y'$ and check whether
\begin{enumerate}[label=(\roman*), align=left]
\item $x \in \mathrm{co} (Y')$, and
\item there is no point $y\in Y\setminus Y'$ such that $\| x-y\| < \dd_{Y'} (x)$.
\end{enumerate}
The point $x$ is critical if, and only if, the conditions (i) and (ii) are verified. Under the general position assumption on $Y$, the procedure we described gives an upper bound of $\sum_{\lini = 0}^{n} \binom{m}{\lini + 1}$ on the number of critical points of $\dd_Y$.

(We continue the discussion of this example below, after we have introduced the notion of nondegenerate critical points, see \cref{example:points2}.)
\end{example}

\subsubsection{Subdifferentials of distance functions via Danskin's theorem}

We now describe the Clarke subdifferential of the distance function in a more direct way. The key point is that the squared distance from a closed set can be written as the minimum of a family of smooth functions. This allows one to apply Danskin's theorem and obtain an explicit formula for the subdifferential, both in the ambient space and after restriction to a smooth submanifold.

We first recall the version of Danskin's theorem that we need.

\begin{proposition}[Danskin's theorem]\label{P:danskin}
Let $W \subseteq \R^m$ be compact, and let $\phi \colon \R^n \times W \to \R$ be continuous. Assume that for every $w \in W$ the function $\phi(-,w)\colon \R^n \to \R$ is differentiable, and that for every $x \in \R^n$ the map
\[
w \mapsto \nabla_x \phi(x,w)
\]
is continuous on $W$. Define
\[
f(x) \coloneqq \max_{w \in W} \phi(x,w).
\]
Then
\[
\partial_x f = \mathrm{co}\left\{ \nabla_x \phi(x,z) \;\middle|\; z \in Z(x)\right\},
\]
where
\[
W(x) \coloneqq \left\{ w \in W \;\middle|\; \phi(x,w)=\max_{w \in W}\phi(x,w)\right\}.
\]
\end{proposition}

\begin{proof}
This is \cite[Prop.~B.22(b)]{bertsekas1999nonlinear}.
\end{proof}

We begin with the ambient case.

\begin{proposition}\label{P:squared-distance-subgradient}
Let $Y \subseteq \R^n$ be closed, and define
\[
\eta_Y(x)\coloneqq \frac{1}{2}\dd_Y(x)^2.
\]
Then, for every $x \in \R^n$,
\[
\partial_x(\eta_Y)= \mathrm{co}\left\{ x-y \;\middle|\; y \in B(x,\dd_Y(x)) \cap Y \right\}.
\]
\end{proposition}

\begin{proof}
Fix $x_0 \in \R^n$. Since $Y$ is closed, for every $x \in \R^n$ the set $B(x,\dd_Y(x)) \cap Y$ is nonempty and compact.

We claim that there exist a neighborhood $U$ of $x_0$ and $R>0$ such that, for every $x \in U$,
\[
B(x,\dd_Y(x)) \cap Y \subseteq B(0,R).
\]
Indeed, choose $U$ bounded. Since $\dd_Y$ is $1$-Lipschitz, it is bounded on $U$, and therefore
\[
R \coloneqq \sup_{x \in U}\|x\|+\sup_{x \in U}\dd_Y(x)
\]
is finite. If $x \in U$ and $y \in B(x,\dd_Y(x)) \cap Y$, then
\[
\|y\| \leq \|y-x\|+\|x\|=\dd_Y(x)+\|x\| \leq R.
\]
This proves the claim.

Set
\[
W \coloneqq Y \cap B(0,R),
\qquad
\phi(x,y)\coloneqq \langle x,y\rangle-\frac{\|x\|^2}{2}-\frac{\|y\|^2}{2}.
\]
For $x \in U$, one has
\[
\eta_Y(x)=\min_{y \in Z}\frac{1}{2}\|x-y\|^2,
\]
or equivalently
\[
-\eta_Y(x)=\max_{y \in Z}\phi(x,y).
\]
Moreover, $W$ is compact, $\phi$ is continuous, for every $y \in W$ the function $\phi(-,y)$ is smooth, and
\[
\nabla_x\phi(x,y)=y-x
\]
depends continuously on $y$. Hence \cref{P:danskin} applies and yields
\[
\partial_x(-\eta_Y)=\mathrm{co}\left\{ y-x \;\middle|\; y \in W(x)\right\},
\]
where, as above,
\[
W(x)=\left\{ y \in W \;\middle|\; \phi(x,y)=\max_{w \in W}\phi(x,w)\right\}.
\]
By construction, the maximizing points are exactly the nearest points to $x$ in $Y$, that is,
\[
W(x)=B(x,\dd_Y(x)) \cap Y.
\]
Therefore
\[
\partial_x(-\eta_Y)=\mathrm{co}\left\{ y-x \;\middle|\; y \in B(x,\dd_Y(x)) \cap Y \right\},
\]
and multiplying by $-1$ gives
\[
\partial_x(\eta_Y)= \mathrm{co}\left\{ x-y \;\middle|\; y \in B(x,\dd_Y(x)) \cap Y \right\}.
\]
\end{proof}

\begin{corollary}\label{propo:subgradient1}
Let $Y\subseteq \R^n$ be closed and let $x\in \R^n \setminus Y$. Then
\[
\partial_x(\dd_Y) = \mathrm{co}\left\{
\left.
\frac{x-y}{\|x-y\|}
\,\right|\,
y \in B(x,\dd_Y(x)) \cap Y
\right\}.
\]
\end{corollary}

\begin{proof}
Since $x \notin Y$, one has $\dd_Y(x)>0$. Moreover,
\[
\eta_Y=\frac{1}{2}\dd_Y^2,
\]
hence at every differentiability point of $\dd_Y$ one has
\[
\nabla \eta_Y=\dd_Y \nabla \dd_Y.
\]
Passing to the Clarke subdifferential at $x$, this gives
\[
\partial_x(\eta_Y)=\dd_Y(x)\,\partial_x(\dd_Y).
\]
Using \cref{P:squared-distance-subgradient}, we obtain
\[
\partial_x(\dd_Y)
=
\frac{1}{\dd_Y(x)}
\mathrm{co}\left\{ x-y \;\middle|\; y \in B(x,\dd_Y(x)) \cap Y \right\}.
\]
Since $\|x-y\|=\dd_Y(x)$ for every $y \in B(x,\dd_Y(x)) \cap Y$, this becomes
\[
\partial_x(\dd_Y) = \mathrm{co}\left\{
\frac{x-y}{\|x-y\|}
\;\middle|\;
y \in B(x,\dd_Y(x)) \cap Y
\right\}.
\]
\end{proof}

We now turn to the restriction of the distance function to a smooth submanifold.

\begin{proposition}\label{propo:subgradient2}
Let $Y \subseteq \R^n$ be closed and let $X \subseteq \R^n$ be a smooth manifold. Then, for every $x \in X \setminus Y$,
\[
\partial_x(\dist)=\mathrm{proj}_{T_xX}\bigl(\partial_x(\dd_Y)\bigr),
\]
where $\mathrm{proj}_{T_xX}\colon \R^n \to T_xX$ denotes the orthogonal projection. More explicitly,
\[
\partial_x(\dist)
=
\mathrm{co}\left\{
\frac{\mathrm{proj}_{T_xX}(x-y)}{\|x-y\|}
\;\middle|\;
y \in B(x,\dd_Y(x)) \cap Y
\right\}.
\]
\end{proposition}

\begin{proof}
Fix $x_0 \in X \setminus Y$, and choose a smooth chart
\[
\psi \colon U \subseteq \R^{\dim X} \to X
\]
such that $\psi(0)=x_0$.

As in the proof of \cref{P:squared-distance-subgradient}, after shrinking $U$ if necessary there exists $R>0$ such that
\[
B(\psi(u),\dd_Y(\psi(u))) \cap Y \subseteq B(0,R)
\qquad \text{for every } u \in U.
\]
Set
\[
W \coloneqq Y \cap B(0,R)
\]
and define
\[
F(u)\coloneqq -\frac{1}{2}\dd_Y|_X(\psi(u))^2.
\]
Then
\[
F(u)=\max_{y \in W}\Phi(u,y),
\]
where
\[
\Phi(u,y)\coloneqq
\langle \psi(u),y\rangle-\frac{\|\psi(u)\|^2}{2}-\frac{\|y\|^2}{2}.
\]
The function $\Phi$ is continuous, for every $y \in Z$ the function $\Phi(-,y)$ is smooth, and
\[
\nabla_u\Phi(u,y)=D\psi(u)^{\top}(y-\psi(u))
\]
depends continuously on $y$. Therefore \cref{P:danskin} applies and gives
\[
\label{Eq:F-subdif-first}
\partial_0F
=
\mathrm{co}\left\{
D\psi(0)^{\top}(y-x_0)
\;\middle|\;
y \in B(x_0,\dd_Y(x_0)) \cap Y
\right\}.
\]
Moreover, by the chain rule, we have
\[
\label{Eq:F-subdif-second}
\partial_0 F = D \psi(0)^\top \partial_{x_0} \left( -\frac{1}{2} \dist^2 \right).
\]
Identifying $T_{x_0} X$ with the image of $D \psi(0)$, Equations \eqref{Eq:F-subdif-first} and \eqref{Eq:F-subdif-second} give the following equality between projections onto $T_{x_0} X$:
\[
\operatorname{proj}_{T_{x_0}} \partial_{x_0} \left( -\frac{1}{2} \dist^2 \right) = \mathrm{co} \left\{
\mathrm{proj}_{T_{x_0}X} (y - x_0)
\;\middle|\;
y \in B(x_0, \dd_Y(x_0)) \cap Y
\right\}.
\]
The subdifferential $\partial_{x_0} \left( -\frac{1}{2} \dist^2 \right)$ is already in $T_{x_0} X$ by construction, so we can omit the projection on the lefthand side. Taking the opposite sign, we get
\[
\partial_{x_0}\left(\frac{1}{2}\dist^2\right)
=
\mathrm{co}\left\{
\mathrm{proj}_{T_{x_0}X}(x_0-y)
\;\middle|\;
y \in B(x_0,\dd_Y(x_0)) \cap Y
\right\}.
\]
Since $x_0 \notin Y$, one has $\dist(x_0)=\dd_Y(x_0)>0$, and therefore
\[
\partial_{x_0}(\dist)
=
\frac{1}{\dist(x_0)}
\partial_{x_0}\left(\frac{1}{2}\dist^2\right).
\]
Using that $\|x_0-y\|=\dd_Y(x_0)$ for every $y \in B(x_0,\dd_Y(x_0)) \cap Y$, we conclude that
\[
\partial_{x_0}(\dist)
=
\mathrm{co}\left\{
\frac{\mathrm{proj}_{T_{x_0}X}(x_0-y)}{\|x_0-y\|}
\;\middle|\;
y \in B(x_0,\dd_Y(x_0)) \cap Y
\right\}.
\]
By \cref{propo:subgradient1}, this is exactly
\[
\partial_{x_0}(\dist)=\mathrm{proj}_{T_{x_0}X}\bigl(\partial_{x_0}(\dd_Y)\bigr),
\]
because orthogonal projection is linear.
\end{proof}

\begin{corollary}\label{coro:projection}
Let $Y \subseteq \R^n$ be closed and let $X \subseteq \R^n$ be a smooth manifold.
\begin{enumerate}
\item If $x \in X \cap Y$, then $x$ is critical for $\dist$.
\item If $x \in X \setminus Y$, then $x$ is critical for $\dist$ if and only if
\[
\partial_x(\dd_Y)\cap N_xX \neq \varnothing,
\]
where $N_xX \coloneqq (T_xX)^\perp$.
\end{enumerate}
\end{corollary}

\begin{proof}
For the first statement, assume that $x \in X \cap Y$.
If $0$ were not in $\partial_x(\dist)$, then we could apply Clarke's inverse function theorem for locally Lipschitz functions \cite[Theorem 1]{clarkeIFT} and deduce that $\dist$ changes signs around $x$, which is impossible as it is always nonnegative. We deduce that $0 \in \partial_x(\dist)$, i.e., $x$ is critical.

Now let $x \in X \setminus Y$. By \cref{propo:subgradient2},
\[
\partial_x(\dist)=\mathrm{proj}_{T_xX}\bigl(\partial_x(\dd_Y)\bigr).
\]
Therefore $x$ is critical for $\dist$ if and only if
\[
0 \in \mathrm{proj}_{T_xX}\bigl(\partial_x(\dd_Y)\bigr).
\]
Since $\partial_x(\dd_Y)$ is convex and $\mathrm{proj}_{T_xX}$ is linear, this is equivalent to the existence of some $v \in \partial_x(\dd_Y)$ such that
\[
\mathrm{proj}_{T_xX}(v)=0.
\]
The latter equality is equivalent to $v \in N_xX$, and the second statement follows.
\end{proof}

\begin{remark}
In \cref{P:squared-distance-subgradient,propo:subgradient1,propo:subgradient2,coro:projection}, the only assumptions are that $Y$ is closed and that $X$ is smooth. In particular, definability and transversality are not used here.
\end{remark}

\begin{example}
As an example, see Figure \ref{F:3D}, where $X$ is a sphere and $Y = \{y_0, y_1\}$ consists of two points in $\R^3$. $X$ intersects $M_Y$ in a circle, and for every point $x$ of this intersection, the subdifferential of $\dd_Y$ at $x$ is the segment $[\frac{x - y_0}{\|x - y_0\|}, \frac{x - y_1}{\|x - y_1\|}]$. However, only when $x$ is $x_0$ or $x_1$ does the normal space $N_x X$ intersect the subdifferential: thus, of all the points of $X$ on the medial axis, only $x_0$ and $x_1$ are critical. Of course, these are not the only critical points; the points minimizing $\dist$ are also critical points, and they do not lie on the medial axis.
\end{example}

\begin{figure}
\centering
\begin{tikzpicture}
\coordinate (c) at (0, 0);
\coordinate (x0) at (2, 0);
\coordinate (x1) at (-2, 0);
\coordinate (y0) at (1, 3);
\coordinate (y1) at (1, -3);

\node (d0) at ($(x0)!-2cm!(y0)$) {};
\node (d1) at ($(x0)!-2cm!(y1)$) {};

\draw[line width=2pt] (-45:2cm) arc (-45:-135:2cm); 

\draw[fill=white] (1.323, 1.5) -- (2, 1.5) -- (4, -1.5) -- (-4, -1.5) -- node[left] {$M_Y$} (-2, 1.5) -- (-1.323, 1.5);
\draw[dashed] (-1.323, 1.5) -- (1.323, 1.5);

\draw[line width=2pt] (x0) arc (0:180:2cm);
\draw[dashed, line width=1pt]
  (x0) arc (0:-48.59:2cm)
  (x1) arc (-180:-131.41:2cm);

\draw[line width=2pt] (x0) arc (0:-180:2cm and 1cm);
\draw[dashed, line width=1pt] (x0) arc (0:180:2cm and 1cm);

\draw[dashed] (y0) -- (y1);

\foreach \x in {0, 45, ..., 315} {
  \coordinate (temp) at (\x:2cm and 1cm);
  \draw[->] (temp) -- ($(temp)!-.2!(c)$);
}
\draw[dashed, line width=1pt] ($(x0) - (2, 0)$) -- ($(x0) + (2, 0)$) node[right] {$N_{x_0} X$};

\draw[->, line width=1pt] (x0) -- (d0) node[below]{$\frac{x_1 - y_0}{\| x_1 - y_0 \|}$};
\draw[->, line width=1pt] (x0) -- (d1) node[above]{$\frac{x_1 - y_1}{\| x_1 - y_1 \|}$};
\draw[line width=1pt] (d0) -- node[right, pos=.8] {$\partial_{x_1} \dd_Y$} (d1);

\filldraw (y0) circle (2pt) node[left] {$y_0$};
\filldraw (y1) circle (2pt) node[right] {$y_1$};
\filldraw (x0) circle (2pt) node[above left] {$x_1$};
\filldraw (x1) circle (2pt) node[right] {$x_0$};
\filldraw (1, 0) circle (2pt);
\filldraw (2.63, 0) circle (2pt);

\node[above left] at (120:2cm) {$X$};
\end{tikzpicture}
\caption{Example in $\R^3$ of the characterization of critical points using normal spaces. $X$ is a sphere and $Y = \{y_0, y_1\}$ consists of two points. The plane represents the medial axis $M_Y$; the small arrows represent normal vectors from various points of $X$; $N_{x_1} X$ (the horizontal dashed line) is the normal space of $X$ at $x_1$.}
\label{F:3D}
\end{figure}

\subsection{The Sard property}\label{sec:sard}

If $X \subseteq \R^n$ is smooth, the fact that the set of critical values of $\dist$ has zero measure  follows from a result of Rifford \cite{Rifford2004}, which generalizes the Morse-Sard theorem for smooth functions to the case of distance functions. In the definable context we can actually use the following result.

\begin{proposition}[{\cite[Theorem 7]{BDLS2005}}]
\label{propo:crit_values_definable}
Let $f \colon X\to \R$ be a locally Lipschitz and definable function. Then $f$ is constant on each component of the set of its critical points and therefore the set of critical values of $f$ is finite. In particular, if $X, Y \subseteq \R^n$ are definable sets with $X$ smooth, the set of critical values of $\dist$ is finite.
\end{proposition}

\begin{remark}\label{remark:needed}
The requirement that $Y$ is a definable set seems necessary to avoid pathological examples. For instance, as stated in \cite[Remark 1.1]{Rifford2004}, while the set of all critical values of the distance function from a compact subset in $\R^3$ is of measure zero \cite{Fu1985}, counterexamples exist in higher dimensions \cite[Example 3]{Ferry1976} (see also \cite{Whitney1935}).
\end{remark}

\subsection{The deformation lemma}

Given $f \colon X \to \R$ and $t \in \R$, we denote by $X^t \coloneqq \{x\in X \mid f(x) \leq t\}$ and in the special case $f = \dist$ we set 
\[
X_Y^{t} \coloneqq \left\{x \in X \,\bigg|\, \dd_Y(x) \leq t\right\}.
\]
We are interested in understanding how the topology changes letting $t \in \R$ vary. This can be measured by the relative cohomology groups $H^*(X_Y^{t + \varepsilon}, X_Y^{t - \varepsilon})$, for $\varepsilon > 0$ sufficiently small. In this article, cohomology groups have coefficients in an arbitrary fixed field.
If  $\{c_0, \ldots, c_k\}$ are the critical values of the distance function, then no topology change happens for the set $X_Y^a$ when $a \in (c_j, c_{j+1})$. This is a consequence of the deformation lemma for locally Lipschitz functions.

\begin{lemma}
\label{lem:deformation}
Let $f \colon X \to \R$ be a proper, locally Lipschitz function. Assume that the interval $[a, b]$ contains no critical values of $f$. Then $X^b$ deformation retracts to $X^a$. 
Moreover, if $f$ is definable, then $X^b$ deformation retracts to $X^a$ even if $a$ is a critical value, as soon as there are no other critical values in $(a, b]$.
\end{lemma}
\begin{proof}
The first part of the statement follows from Clarke's implicit function theorem \cite[Theorem 1]{clarkeIFT}, see \cite[Proposition 1.2]{APS1997}. The second part follows from the first  part and the definable triviality \cite[Chapter 9, Theorem 1.2]{vdd}.
\end{proof}

\subsection{Passing a critical value: general case}\label{sec:triviality}
Observe first that, if $X, Y$ are definable with $X$ smooth, when $t \in \R$ passes a critical value, the topology of $X_Y^t$ changes in a controlled way. This follows again from the definable triviality and is an application of the next result. 
\begin{proposition}
\label{propo:criticalpass1}
Let $X, Y$ be definable with $X$ smooth and let $f \colon X \to \R$ be a Lipschitz, definable function. For every critical value $c \in \mathbb{R}$ there exists a definable set $K(c)$ such that for all $\varepsilon > 0$ small enough
\[
H^*(X_Y^{c + \varepsilon}, X_Y^{c - \varepsilon}) \cong \widetilde{H}^*(K(c)).
\]
\end{proposition}
\begin{proof}
This is clear, since for $\varepsilon > 0$ small enough all the pairs $(X_Y^{c + \varepsilon}, X_Y^{c - \varepsilon})$ are homotopy equivalent (by definable triviality) and therefore one can define $K(c) \coloneqq X_Y^{c + \varepsilon}/X_Y^{c - \varepsilon}$. Since definable functions can be triangulated, then up to a definable homeomorphism, $X_Y^{c - \varepsilon}$ is a subcomplex of $X_Y^{c + \varepsilon}$ and therefore $K(c)$ is definable.
\end{proof}

\begin{remark}More generally, an analogue of the previous proposition still holds true for a proper definable function $f \colon X\to \R$ (without the assumption that it is locally Lipschitz). 
In fact, by definable triviality \cite[Chapter 9, Theorem 1.2]{vdd},  one can partition $\R$ into finitely many points and intervals such that on each interval $f$ is a trivial fibration. The sublevel sets $X^t$ of $f$ can only change when we pass the points -- we don't call these critical points (because $f$ might not be locally Lipschitz), but the relative cohomology is still that of a definable set.
\end{remark}

\subsection{Passing a critical value: nondegenerate case}

In order to say something more on the change in topology  of $X_Y^t$ when $t$ passes a critical value, in a way that is similar to Morse theory, we need to introduce the notion of \emph{nondegenerate} critical point.
This notion can be introduced at least for a special class of locally Lipschitz functions called \emph{continuous selections}.

\begin{definition}[Continuous selection]
Let $X$ be a smooth manifold and $\{f_0, \ldots, f_m\}\subset \mathcal{C}^2(X, \R)$. We will say that $f \colon X\to \R$ is a \emph{continuous selection} from $\{f_0, \ldots, f_m\}$ if $f$ is continuous and for every $x\in X$ there exists $i\in \{0, \ldots, m\}$ such that $f(x)=f_i(x).$
In this case we will write $f\in \mathrm{CS}(f_0, \ldots, f_m)$.
\end{definition}
If $f\in \mathrm{CS}(f_0, \ldots, f_m)$, for $i\in \{0, \ldots, m\}$ we denote by $S_i\subseteq X$  the set
\[
S_i\coloneqq\left\{x\in X\,\bigg|\, f(x)=f_i(x)\right\}.
\]
Given $x\in X$ the \emph{effective index set} of $f\in \mathrm{CS}(f_0, \ldots, f_m)$ at $x\in X$ is defined by
\[
I(x)\coloneqq\left\{i\in \{0, \ldots, m\}\,\bigg|\, x\in \overline{\operatorname{int}(S_i)}\right\}.
\]
Note that continuous selections are locally Lipschitz \cite{APS1997}, so we can talk about their critical points using the definitions above. Again, below we assume that $X$ is endowed with some Riemannian structure, so as to identify its tangent and cotangent bundles.

\begin{definition}[Nondegenerate critical points of continuous selections]
\label{D:nondegenerate-crit-pt}
Let $f\in \mathrm{CS}(f_0, \ldots, f_m)$. A critical point $x\in C(f)$ is said to be \emph{nondegenerate} (with respect to $f_0, \ldots, f_m$) if the following conditions are satisfied:
\begin{enumerate}
\item For every $i\in I(x)$ the set of differentials $\{D_x f_j \mid j\in I(x)\setminus \{i\}\}$ is linearly independent.
\item Denote by $(\lambda_i)_{i\in I(x)}$ the unique list of nonnegative real numbers such that 
\[
\sum_{i \in I(x)} \lambda_i = 1
\quad \textrm{and} \quad
\sum_{i \in I(x)} \lambda_i D_x f_i = 0
\]
and by $\lambda f \colon X\to \R$ the function defined by $\lambda f(z)\coloneqq\sum_{i\in I(x)}\lambda_if_i(z).$ Then the second differential of $\lambda f$ is nondegenerate on 
\[
V(x) \coloneqq \bigcap_{i\in I(x)}\ker(D_xf_i).
\]
\end{enumerate}
We denote by $\lini(x)\coloneqq\#I(x)-1$ and call this number the \emph{piecewise linear index}, and by $\quadi(x)$ the negative inertia index of the restriction of the second differential of $\lambda f$ to $V(x)$. We call $\quadi(x)$ the \emph{quadratic index}. See Figure \ref{F:nondegenerate} for an illustration of nondegenerate and degenerate critical points of $\dist$, and Figure \ref{F:indices} for an illustration of different indices for $\dist$.

\begin{figure}
\captionsetup[subfigure]{justification=centering}
\renewcommand\thesubfigure{\scshape\alph{subfigure}}
\begin{subfigure}[t]{.3\textwidth}
\centering
\begin{tikzpicture}
\coordinate (x) at (0, 0);
\coordinate (y1) at (0:1.5cm);
\coordinate (y2) at (120:1.5cm);
\coordinate (y3) at (-120:1.5cm);

\draw (x) circle (1.5cm);

\filldraw[line width=2pt, fill=white!80!black] ($(x)!1cm!60:(y1)$) -- ($(x)!1cm!60:(y2)$) -- ($(x)!1cm!60:(y3)$) -- cycle;
\node[yshift=3ex] at ($(x)!1cm!60:(y2)$) {$\partial_x f$};

\filldraw (x) circle (2pt) node[above]{$x$};
\filldraw (y1) circle (2pt) node[right]{$y_0$};
\filldraw (y2) circle (2pt) node[above]{$y_1$};
\filldraw (y3) circle (2pt) node[below]{$y_2$};
\end{tikzpicture}
\caption{nondegenerate critical point}
\label{F:sub-nondegen}
\end{subfigure}
\begin{subfigure}[t]{.3\textwidth}
\centering
\begin{tikzpicture}
\coordinate (x) at (0, 0);
\coordinate (y1) at (1.5, 0);
\coordinate (y2) at (0, 1.5);
\coordinate (y3) at (-1.5, 0);
\coordinate (y4) at (0, -1.5);

\draw (x) circle (1.5cm);

\filldraw[line width=2pt, fill=white!80!black] (1, 0) -- (0, 1) -- (-1, 0) -- (0, -1) -- cycle;
\node[yshift=3ex] at (-1, 0) {$\partial_x f$};

\filldraw (x) circle (2pt) node[above]{$x$};
\filldraw (y1) circle (2pt) node[right]{$y_0$};
\filldraw (y2) circle (2pt) node[above]{$y_1$};
\filldraw (y3) circle (2pt) node[left]{$y_2$};
\filldraw (y4) circle (2pt) node[below]{$y_3$};
\end{tikzpicture}
\caption{degenerate critical point for the first condition}
\label{F:sub-degen1}
\end{subfigure}
\begin{subfigure}[t]{.3\textwidth}
\centering
\begin{tikzpicture}
\coordinate (x) at (-.2, -.4);
\coordinate (y1) at (-1.4, .2);

\draw[line width=2pt] (-2.25, -1.5) -- (-.5, 2) node[right]{$Y$};
\draw[line width=2pt] (-1, -2) -- (.75, 1.5) node[right]{$X$};

\filldraw (x) circle (2pt) node[left]{$x$};
\filldraw (y1) circle (2pt) node[left]{$y_0$};
\end{tikzpicture}
\caption{degenerate critical point for the second condition}
\label{F:sub-degen2}
\end{subfigure}
\caption{Examples of critical points of $f = \dist$ and their degeneracy. (\subref{F:sub-nondegen}) and (\subref{F:sub-degen1}): $X$ is the plane $\R^2$ and $Y$ is a finite set of points in the plane. (\subref{F:sub-degen2}): $X$ and $Y$ are lines.}
\label{F:nondegenerate}
\end{figure}

\begin{figure}
\captionsetup[subfigure]{justification=centering}
\renewcommand\thesubfigure{\scshape\alph{subfigure}}
\begin{subfigure}{.3\textwidth}
\centering
\begin{tikzpicture}
\coordinate (x) at (0, 0);
\coordinate (y1) at (-1, 0);
\coordinate (y2) at (1, 0);

\draw[line width=2pt] (-1, -2) -- (1, 2) node[right]{$X$};

\filldraw (x) circle (2pt) node[left]{$x$};
\filldraw (y1) circle (2pt) node[left]{$y_1$};
\filldraw (y2) circle (2pt) node[right]{$y_2$};
\end{tikzpicture}
\caption{indices $(1, 0)$}
\label{F:sub-10}
\end{subfigure}
\begin{subfigure}{.3\textwidth}
\centering
\begin{tikzpicture}
\coordinate (x) at (0, 0);
\coordinate (y1) at ($(x)!1cm!90:(1, 2)$);

\draw[line width=2pt] (-1, -2) -- (1, 2) node[right]{$X$};

\filldraw (x) circle (2pt) node[left]{$x$};
\filldraw (y1) circle (2pt) node[left]{$y_1$};
\end{tikzpicture}
\caption{indices $(0, 0)$}
\label{F:sub-00}
\end{subfigure}
\begin{subfigure}{.3\textwidth}
\centering
\begin{tikzpicture}
\coordinate (x) at (0, 0);
\coordinate (y1) at ($(x)!1cm!90:(1, 2)$);
\coordinate (c) at ($(x)!-.5cm!90:(1, 2)$);

\draw[line width=2pt] (-1, -2) -- (1, 2) node[right]{$X$};
\draw[line width=2pt] (c) circle (1.5);
\node[above left] at ($(x)!-2cm!90:(1, 2)$) {$Y$};

\filldraw (x) circle (2pt) node[left]{$x$};
\filldraw (y1) circle (2pt) node[left]{$y_1$};
\end{tikzpicture}
\caption{indices $(0, 1)$}
\label{F:sub-01}
\end{subfigure}
\caption{Examples of $\dist$ and $x \in X$ with different piecewise linear and quadratic indices, written as $(\lini(x), \quadi(x))$. (\subref{F:sub-10}) and (\subref{F:sub-00}): $X$ is a line and $Y$ is a set of points. (\subref{F:sub-01}): $X$ is a line, $Y$ is a circle, and $y_1$ is the closest point of $Y$ to $x$.}
\label{F:indices}
\end{figure}

Clearly, in the case $X$ is endowed with a Riemannian metric, the above conditions can be rephrased using the gradients of the $f_i$ instead of their differential. Moreover, the nondegeneracy of the second differential of $\lambda f$ on $V(x)$ is equivalent to the nondegeneracy of the restriction of the corresponding quadratic form on it.
\end{definition}

\begin{remark}
\label{remark:restriction}
An important remark concerns the reformulation of these conditions in the case $X$ is a submanifold of another manifold $N$ and we are considering the restriction of a continuous selection $f=g|_X$ where $g\in \mathrm{CS}(g_0, \ldots, g_m)$ with $g_i\in \mathcal{C}^2(N, \R)$, so that
$$f_j=g_j|_{X}, \quad j=0, \ldots, m.$$  Suppose that $X$ is given as the zero set of a regular equation $p=0$, with $p:N\to \R$ a smooth function (in \cref{S:Genericity-algebraic} we are interested in the case $N=\R^n$ and $p$ is a polynomial). Then, for all $j=0, \ldots, m$,
$$D_xf_j=D_xg_j|_{T_xX}=D_xg_j|_{\ker(D_xp)}.$$
Rewriting the above conditions with this notation, we see that $x\in X$ is a nondegenerate critical point of $f=g|_X$ if and only if
\begin{enumerate}
\item for every $i\in I(x)$ the set of differentials $\{D_x g_j|_{\ker(D_xp)} \mid j\in I(x)\setminus \{i\}\}$ is linearly independent;
\item denote by $\{\mu, (\lambda_i)_{i\in I(x)}\}$ the unique list of  numbers such that the $\lambda_i$ are nonnegative, $\mu\neq 0$,
\[
\sum_{i \in I(x)} \lambda_i = 1,
\quad \textrm{and} \quad
\mu D_xp+\sum_{i \in I(x)} \lambda_i D_x g_i = 0,
\]
and by $\lambda f \colon X\to \R$ the function defined by $\lambda f(z)\coloneqq\sum_{i\in I(x)}\lambda_if_i(z).$ Then the second differential of $\lambda f$ is nondegenerate on 
\[
V(x) \coloneqq \bigcap_{i\in I(x)}\ker(D_xf_i|_{\ker(D_xp)})=\ker(D_xp)\cap\left( \bigcap_{i\in I(x)}\ker(D_xf_i)\right).
\]
\end{enumerate}
In order to compute the second differential of $\lambda f$ on $V(x)$ using the original functions, we proceed as follows. We consider the function
\[
\widetilde{\lambda g} \coloneqq \mu p+\sum_{i\in I(x)}\lambda_ig_i \colon N \to \R.
\]
Observe that, since $X=\{p=0\}$, then $\widetilde{\lambda g}|_{X}=\lambda f.$ Moreover, $\widetilde{\lambda g}$ has a critical point at $x\in N$ and therefore the second differential of its restriction to $X$ equals the restriction of its second differential to $T_xX$. This means that the condition that the second differential of $\lambda f$ is nondegenerate on $V(x)$ is equivalent to the condition that
$$\mu D_x^2p+\sum_{i\in I(x)}\lambda_iD_x^2g_i\quad \textrm{is nondegenerate on $V(x)$}.$$
\end{remark}

\begin{remark}
\label{rmk:normalformdist1}
Continuous selections have normal forms at their critical points, in a fashion similar to the Morse Lemma \cite[Lemma~2.2]{milnor1963morse}. More precisely, if $x$ is a nondegenerate critical point of $f\in \mathrm{CS}(f_0, \ldots, f_m)$ with piecewise linear index $\lini = \lini(x)$ and quadratic index $\quadi = \quadi(x)$, then there exists a neighborhood $U$ of $x$ and a locally Lipschitz homeomorphism $\psi\colon\R^\lini \times \R^{n - \lini} \to U$ such that
\be
\label{eq:normal}
f(\psi(t_1, \ldots, t_\lini, t_{\lini+1}, \ldots, t_{n}))=f(x)+\ell(t_1, \ldots, t_\lini)-\sum_{j=\lini+1}^{\lini + \quadi} t_{j}^2 + \sum_{j = \lini + \quadi + 1}^n t_j^2,
\ee
where $\ell(t) \in \mathrm{CS}(t_1, \ldots, t_\lini, -(t_1 + \cdots + t_\lini))$, see \cite[Theorem 3.3]{JongenPallaschke1988}. Note in particular that nondegenerate critical points of continuous selections are isolated.
\end{remark}

In the special case of a continuous selection of the form $$f=\min\{f_0, \ldots, f_m\},$$ if $x$ is a nondegenerate critical point with indices $\lini = \lini(x)$ and $\quadi = \quadi(x)$, then 
the piecewise linear continuous selection $\ell$ in the normal form  (\ref{eq:normal}) is $\ell(t_1, \ldots, t_\lini) = \min \{ t_1, \ldots, t_\lini, -(t_1 + \cdots + t_\lini)\}$, see \cite[Theorem~3.3]{JongenPallaschke1988}.
In particular, in this case, if $f$ is also definable, then by \cite[Sect.~4]{APS1997} it follows the set $K(c)$ defined in \cref{propo:criticalpass1} depends only on $\lini$ and $\quadi$.

\begin{proposition}
\label{propo:passmin}
Let $f$ be a continuous definable selection of the form $f=\min\{f_0, \ldots, f_m\}$. Assume that all the critical points of $f$ at level $c\in \R$ are nondegenerate and denote them by $\{x_1, \ldots, x_\nu\}$. Then the set $K(c)$ from \cref{propo:criticalpass1} equals
$$K(c) = \bigsqcup_{j=1}^\nu S^{\lini(x_j) + \quadi(x_j)}.$$\
\end{proposition}
\begin{proof}This follows from \cite[Sect.~4]{APS1997}, analyzing the conditions of \cite[Theorem 4.2]{APS1997} case by case.
\end{proof}

{ We will see in \cref{P:regular-value} a sufficient condition for the distance function $\dd_Y$ being of the form $\min\{f_0, \ldots, f_m\}$ in a neighborhood of a point $x$.

\begin{remark}
\label{rmk:normalformdist}
Let $x$ be a nondegenerate critical point with indices $\lini = \lini(x)$ and $\quadi = \quadi(x)$ of a continuous selection $f=\min\{f_0, \ldots, f_m\}$. Then $f$ is topologically equivalent to a quadratic form \cite[Corollary~3.4]{JongenPallaschke1988}:  
there exists a neighborhood $U$ of $x$ and a homeomorphism $\psi\colon\R^n \to U$ such that
\be
f(\psi(t_1, \ldots, t_{n})) = f(x) - \sum_{j=1}^{\lini + \quadi}t_{j}^2+\sum_{j=\lini + \quadi + 1}^n t_j^2 .
\ee
(Note that the conclusions of \cref{propo:passmin} can also be derived from this statement, following the standard argument from Morse theory.) 
We observe that a normal form of $f$ of this type can be also obtained using the theory of Min-type functions, see \cite[Sect.~3.2]{gershkovich1997morse}.
\end{remark}
}

\subsubsection{Nondegenerate critical points of distance functions}
In order to use the notion of nondegenerate critical point for $\dist$ introduced above, we need to make this function a continuous selection of \emph{finitely many} $\mathcal{C}^2$ functions, at least nearby a
critical point. In the case $Y$ is smooth, this is done as follows.

Denote by $\varphi \colon NY\to \R^n$ the restriction of the exponential map of $\R^n$ to the normal bundle of $Y$. Note that we can think of elements in $NY$ as pairs $(y, v)\in \R^{2n}$ with $y\in Y$ and $v\in N_yY$ and $\varphi(y,v)=y+v.$

\begin{proposition}
\label{P:regular-value}
Suppose that $Y$ is smooth and let $x \in \R^n$ be a regular value of $\varphi \colon NY \to \R^n$. Then \begin{enumerate}
\item \label{I:finite} $B(x, \dd_Y(x)) \cap Y$ is a finite set $\{y_0, \ldots, y_k\}$,
\end{enumerate}
and for all $\varepsilon > 0$ small enough, there exists $\delta >0$ such that
\begin{enumerate}[resume]
\item \label{I:C2} for every $i \in \{0, \ldots, k\}$, the function $f_i \coloneqq \dd_{B(y_i, \varepsilon) \cap Y}$ is $\mathcal{C}^2$ on $B(x, \delta)$,
\item \label{I:CS} the function $\dd_Y|_{B(x, \delta)}$ is a continuous selection of the $f_i$: specifically, $\dd_Y|_{B(x, \delta)} = \min_i f_i$.
\end{enumerate}
\end{proposition}

\begin{proof}
Let $y \in B(x, \dd_Y(x)) \cap Y$. Then $(y, x - y) \in N Y$ and $\varphi(y, x - y) = x$. Since $x$ is a regular value of $\varphi$, no point $(y', v')$ near $(y, x - y)$ can have image $x$ by $\varphi$. In particular, this means that $y$ is an isolated point in $B(x, \dd_Y(x)) \cap Y$. Since the set $B(x, \dd_Y(x)) \cap Y$ is compact, we deduce that $B(x, \dd_Y(x)) \cap Y$ is a finite set, which we write as $\{y_0, \ldots, y_k\}$. This proves \eqref{I:finite}.

Since $x$ is a regular value of $\varphi$, we have, for all $i \in \{0, \ldots, k\}$, neighborhoods $U_i$ of $x$ in $\R^n$, $V_i$ of $y_i$ in $Y$, and $W_i$ of $x - y_i$ in $\R^{n - \dim Y} \cong N_{y_i} Y$ such that $V_i \times W_i \subseteq N Y$ (up to trivialization) and $\varphi_i \coloneqq \varphi|_{V_i \times W_i} \colon V_i \times W_i \to U_i$ is a diffeomorphism. In particular, we can choose these neighborhoods such that every point of $U_i$ is a regular value of $\varphi$ and such that $f = \min_i \dd_{V_i}$ on $U \coloneqq \bigcap_i U_i$. We define the function $\psi_i \colon U_i \to V_i$ by $\psi_i(z) \coloneqq \textnormal{proj}_1 (\varphi_i^{-1}(z))$. In particular, $\psi_i(x) = y_i$. Note that $\psi_i$ is a $\mathcal{C}^2$ function. Then, for all $z$ in $U_i$, $\dd_{V_i}(z) = \|\psi_i(z) - z\|$, and so the function $\dd_{V_i}$ is also $\mathcal{C}^2$.

For the choice of $V_i$, we can choose a ball around $y_i$ with radius $\varepsilon_i > 0$, and set $\varepsilon = \min_i \varepsilon_i$. This restricts the choice of the $U_i$, but we can always take $\delta$ such that $B(x, \delta) \subseteq U$. Then the functions $f_i \coloneqq \dd_{V_i}$ satisfy \eqref{I:C2} and \eqref{I:CS}.
\end{proof}

Related results to \cref{P:regular-value} can be found in \cite[Sect.~3]{song2023generalized}.

\begin{example}\label{example:nonCS}
The conclusions of \cref{P:regular-value} are not valid if we drop the assumption that $x$ is a regular value of the normal exponential map. For instance, in the case of a smooth parabola $Y=\{x_2=x_1^2\}$ in $\R^2$, the set $M_Y$ equals the half line $\{(0, t), t > 1\}$. The point $x=(0,1) \in \overline{M_Y}$ has a unique closest point in $Y$ (the origin), but it is not a regular value of the normal exponential map. Moreover every point $x'$ of the form $x'=(0, t)$ with $t>1$ has two closest points in $Y$ and so the function $\dd_Y$ is not differentiable at $x$, and therefore it cannot be $\mathcal{C}^2$, contradicting \cref{P:regular-value}\eqref{I:C2}.
\end{example}

\subsection{Morse theory of distance functions}

Let us go back now to the problem of understanding the contribution of a nondegenerate critical point of $\dist$ to the homology of $X$. To start with, we recall the following result from semialgebraic geometry, see \cite[Theorem 1.42]{lecturenotesantonio}. 
The result shows that compact definable sets have definable mapping cylinder neighborhoods (the proof therein works in the definable framework).

\begin{theorem}
\label{thm:mappingcylinder}
Let $S$ be a compact definable set and $f \colon S \to [0, \infty)$ be a continuous definable function. Denote by $Z \coloneqq f^{-1}(0)$ the zero set of $f$. Then there exists $\varepsilon_f > 0$ such that for every $\varepsilon \in (0, \varepsilon_f)$ the set $\{f \leq \varepsilon\}$ is a mapping cylinder neighborhood of $Z$ in $S$. In particular, the inclusions $Z\hookrightarrow \{f < \varepsilon\} \hookrightarrow \{f \leq \varepsilon\}$ are homotopy equivalences.
\end{theorem}

Going back to our setting, note that $X\cap Y$ is always critical and we have the following corollary.

\begin{corollary}
Let $Y$ be closed and definable and $X$ smooth and definable in $\R^n$. Then, for all $R > 0$, there exists $\varepsilon > 0$ such that, for all $t \in (0, \varepsilon)$, $X_Y^t \cap B(0, R)$ deformation retracts to $X_Y^0 \cap B(0, R)$.
\end{corollary}

\begin{proof}
Given $R>0$ consider the compact definable set $S \coloneqq X \cap B(0, R)$  together with the continuous definable function $f \colon S \to [0, \infty)$ given by $\alpha \coloneqq \dd_Y|_S.$ Then the conclusion follows from the definable version of \cref{thm:mappingcylinder}.
\end{proof}

The following definition is useful.

\begin{definition}\label{def:setnondegcritpts} Let $Y\subseteq \R^n$ be closed and definable and $X$ be a smooth manifold. We say that a point $x$ of $X$ is a \emph{nondegenerate critical point of $\dist$} if $\dist$ is a continuous selection at $x$ and $x$ is a nondegenerate critical point for the continuous selection. For every $\lini, \quadi \in \N$ we denote by $C_{\lini, \quadi}(X,Y)$ the set of nondegenerate critical points of $\dist$ with piecewise linear index $\lini$ and quadratic index $\quadi$ and which are not in $Y$.
\end{definition}

The next result is a specialization of \cref{propo:criticalpass1,propo:passmin} for nondegenerate critical points of distance functions.
\begin{corollary}
\label{propo:criticalpass2dist}
Let $Y\subseteq \R^n$ be closed and definable and $X\subseteq \R^n$ be a smooth manifold. Let $c\neq 0$ be a critical value of $\dist$ such that 
$C(\dist)\cap (\dist)^{-1}(c)=\{x_1, \ldots, x_\nu\}$ consists only of nondegenerate critical points. For every $j=1, \ldots, \nu$ denote by $(\lini_j, \quadi_j)$ the pair of numbers such that $x_j\in C_{\lini_j, \quadi_j}(X, Y).$
Then for $\varepsilon>0$ sufficiently small, 
\[
H^*(X^{ c+\varepsilon}_Y, X^{c-\varepsilon}_Y)\cong \bigoplus_{j=1}^\nu\widetilde{H}^*(S^{\lini_j+\quadi_j}).
\]
\end{corollary}

\begin{definition}[Nondegenerate pairs]
\label{D:nondegenerate-pairs}
We say that a pair $(X,Y)$ is a \emph{nondegenerate pair} if all critical points of $\dist$ not in $Y$ are nondegenerate.
\end{definition}

For a pair of spaces $Z'\subseteq Z$, let $b_i (Z)$ denote the dimension of $H^i (Z)$ as a vector space over the fixed field of coefficients, and let $b_i (Z,Z')$ denote the dimension of $H^i (Z,Z')$. In our setting, strong Morse inequalities can be stated as follows.  

\begin{proposition}[strong Morse inequalities]
\label{propo:strongMorse}
Let $Y\subseteq \R^n$ be a closed and definable set and $X\subseteq \R^n$ be a smooth, compact and definable manifold such that $(X, Y)$ is a nondegenerate pair. 
Then, for any integer $\lambda \ge 0$, 
\[
\sum_{i=0}^{\lambda} (-1)^{i+\lambda} b_i (X) \le \sum_{i=0}^{\lambda} (-1)^{i+\lambda} \left( b_i (X\cap Y) + \sum_{\lini+\quadi=i} \# C_{\lini,\quadi}(X,Y) \right) .
\]
\end{proposition}

\begin{proof}
\label{prop:Sardproperty}
Let $0=c_0 < c_1 < \cdots < c_m$ denote the critical values of $\dist$, which are finitely many by \cref{propo:crit_values_definable}. Choose real numbers $\{d_j\}_{j=0}^{m+1}$ with $d_0<c_0$, $c_m < d_{m+1}$, and $c_{j-1} < d_j < c_j$ for all $j\in \{1,\ldots ,m\}$. Proceeding exactly as in \cite[Sect.~5]{milnor1963morse} one obtains, for every integer $\lambda \ge 0$, the inequality
\[
\sum_{i=0}^{\lambda} (-1)^{i+\lambda} b_i (X_Y^{d_{m+1}}) \le \sum_{j=0}^{m} \sum_{i=0}^{\lambda} (-1)^{i+\lambda} b_i (X_Y^{d_{j+1}},X_Y^{d_{j}}).
\]
On the left-hand side, we observe that $b_i (X_Y^{d_{m+1}})=b_i (X)$.
On the right-hand side, we can rearrange the sum as 
\[
\sum_{i=0}^{\lambda} (-1)^{i+\lambda} \left( \sum_{j=0}^{m}  b_i (X_Y^{d_{j+1}},X_Y^{d_{j}}) \right)
\]
and observe that $\sum_{j=0}^{m}  b_i (X_Y^{d_{j+1}},X_Y^{d_{j}}) = b_i (X\cap Y) + \sum_{\lini+\quadi=i} \# C_{\lini,\quadi}(X,Y)$ by \cref{propo:criticalpass2dist}, noting that the term $b_i (X\cap Y)$ comes from crossing the $0^{\text{th}}$ critical value. 
\end{proof}

The weak Morse inequalities
\[
b_{\lambda} (X) \le b_{\lambda} (X \cap Y) + \sum_{\lini + \quadi = \lambda} \# C_{\lini, \quadi}(X,Y)
\]
are obtained from \cref{propo:strongMorse} by adding the strong inequalities for $\lambda$ and $\lambda -1$.
Writing $\chi(Z) \coloneqq \sum_{i\ge 0} (-1)^i b_i (Z)$, we obtain the following interesting result that relates the Euler characteristics of $X$ and $Y$ to the critical points of the corresponding relative distance functions.

\begin{corollary}\label{coro:duality}
Let $Y\subseteq \R^n$ be a closed and definable set and $X\subseteq \R^n$ be a smooth, compact, and definable manifold such that $(X, Y)$ is a nondegenerate pair. Then
\[
\chi(X\cap Y)+\sum_{\lini, \quadi \geq 0}(-1)^{\lini + \quadi}\#C_{\lini, \quadi}(X, Y)=\chi(X).
\]
Furthermore, if $Y$ is smooth, $X$ is closed and definable, and $(Y,X)$ is a nondegenerate pair, we get the following equation:
\begin{equation}
\label{E:chi}
\chi(Y)+\sum_{\lini, \quadi\geq 0}(-1)^{\lini + \quadi}\#C_{\lini, \quadi}(X, Y)=\chi(X)+\sum_{\lini, \quadi\geq 0}(-1)^{\lini + \quadi}\#C_{\lini, \quadi}(Y, X).
\end{equation}
\end{corollary}
\begin{proof} As in the classical case, the first equality follows from \cref{propo:strongMorse} by comparing the strong Morse inequalities for large enough values $\lambda$ and $\lambda +1$. The second equality follows from the first one applied to the pairs $(X,Y)$ and $(Y,X)$.
\end{proof}

\begin{remark}The conclusions of \cref{propo:strongMorse} are still valid if $X=\R^n$ and $Y$ is compact, since this guarantees that the distance function from $Y$ is proper.
\end{remark}

\begin{example}
Let $Y=\{y\}$ with $y \notin X$. Then the piecewise linear index of $\dist$ at each critical point is zero, and the quadratic index is the standard Morse index, and so the left-hand side of \eqref{E:chi} equals $1 + \sum_{\quadi} (-1)^\quadi \#C_{0,\quadi}(\dd_{\{y\}}|_X)$. On the other hand, $\dd_X|_{\{y\}}$ is a constant function, with critical point $\{y\}$ of total index zero, and so the right hand side equals $\chi(X) + 1$. In this case, this formula gives back the Euler Characteristic of $X$ as the alternating sum of the number of critical points of quadratic index $\quadi$.
\end{example}

\begin{example}[Distance function from a finite set, continuation of \cref{ex:finiteY}]\label{example:points2}
In this example we provide further details on nondegenerate critical points of $\dd_Y$, with $Y=\{y_0,\ldots ,y_m\}\subset \mathbb{R}^n$, and review abstract and geometric complexes capturing the topology of sublevels sets of $\dd_Y$.

For the generic choice of the points $\{y_0, \ldots, y_m\}$, all critical points of $\dd_Y$ are nondegenerate and the critical points can be found with the procedure described in \cref{ex:finiteY}. 
In view of \cref{D:nondegenerate-crit-pt,sec:related}, we see that the piecewise linear index $\lini$ of a critical point is not an analogue in the context of distance functions of the classical Morse index; in this case, however, as the quadratic index $\quadi$ is always zero, the piecewise linear index is what determines the contribution to the topology change.
More precisely, using \cref{propo:criticalpass2dist}, we see that if $c\neq 0$ is a critical value of $\dd_Y$ such that 
$C(\dd_Y)\cap \dist^{-1}(c)=\{x_1, \ldots, x_\nu\}$ consists only of nondegenerate critical points, then for $\varepsilon>0$ sufficiently small, 
\[
H^*((\R^n)^{ c+\varepsilon}_Y, (\R^n)^{c-\varepsilon}_Y)\cong \bigoplus_{j=1}^\nu\widetilde{H}^*(S^{\lini_j}).
\]

The critical points of $\dd_Y$ can be characterized via the interplay of the Voronoi diagram and the Delaunay tessellation, see e.g.\ \cite{buchin2008recursive}. 
The \emph{Delaunay tessellation} of a finite set $Y\subset \mathbb{R}^n$ is the dual of the Voronoi diagram of $Y$. It is a polytopal complex (i.e., a polyhedral complex such that all cells are bounded) embedded in $\mathbb{R}^n$ and covering the convex hull of $Y$. If $Y$ is in general position, it is furthemore a (geometric) simplicial complex. The set of vertices of the Delaunay tessellation is $Y$. An $n$-dimensional Delaunay cell is the convex hull of $n+1$ or more points in $Y$ if the intersection of their corresponding Voronoi cells is a vertex of the Voronoi diagram. Delaunay cells and their faces constitute the Delaunay tessellation. If $x$ is a nondegenerate critical point of $\dd_Y$, then the convex hull $\mathrm{co}(B(x,\dd_Y)\cap Y)$ is the Delaunay face dual to the lowest dimensional Voronoi face containing $x$. In the rest of this example, let us assume $Y$ to be in general position and all critical points of $\dd_Y$ to be nondegenerate. Then, the critical points of $\dd_Y$ are exactly the intersection between the Voronoi faces and their dual Delaunay faces \cite{buchin2008recursive}.

For $n=2$, the topology of the sublevels sets of the distance function is described in details in \cite{siersma1999voronoi}. For an arbitrary $n$, some abstract and geometric complexes are known to have the same homotopy type of the sublevel sets of $\dd_Y$. 

The \emph{\v{C}ech complex} of $Y$ at scale $t$, denoted by  $\operatorname{\check{C}}_t(Y)$, is the abstract simplicial complex defined as the nerve of the collection of balls $\{B(y_i,t)\}_{i=0}^{m}$ in $\mathbb{R}^n$. By the nerve lemma \cite{borsuk1948imbedding,bauer2023unified}, 
$\operatorname{\check{C}}_t(Y)$ is homotopy equivalent to the sublevel set $(\mathbb{R}^n)_Y^t = \bigcup_{i=0}^{m} B(y_i,t)$ of $\dd_Y$. 

The subcomplex $\mathrm{Del}_t(Y)$ of $\operatorname{\check{C}}_t(Y)$ given by the nerve of the collection $\{B(y_i,t)\cap V(y_i)\}_{i=0}^{m}$ of balls intersected with the respective Voronoi cells is called \emph{$\alpha$-shape} or \emph{Delaunay complex} of $Y$ at scale $t$ \cite{edelsbrunner1995union}. If $Y$ is in general position, $\mathrm{Del}_t(Y)$ is embedded in $\mathbb{R}^n$ as a subcomplex of the Delaunay tessellation, which is a simplicial complex in $\mathbb{R}^n$ in this case. The \v{C}ech and Delaunay complexes at scale $t$ are homotopy equivalent \cite{edelsbrunner1995union}, and the latter can be obtained from the former via a sequence of simplicial collapses \cite{bauer2017morse}.

The \emph{flow complex} of $Y$ \cite{giesen2008flow,buchin2008recursive} is a polytopal complex in $\mathbb{R}^n$ defined using the critical points of $\dd_Y$. A vector field on $\mathbb{R}^n$ can be defined using Voronoi faces and their dual Delaunay faces. The critical points of $\dd_Y$ are exactly the fixed points of a flow associated with the vector field. 
The stable manifolds of the critical points of $\dd_Y$ with respect to the flow define the flow complex, see \cite[Sect.~3]{buchin2008recursive}; note that here we consider an immediate adaptation of the definition, using the function $\dd_Y$ instead of $\dd^2_Y$. The flow complex has exactly one cell for each critical point of $\dd_Y$. 

Polytopal subcomplexes of the flow complex can be defined for each scale $t$. The underlying spaces of these complexes are called \emph{flow shapes} of $Y$ at scale $t$, denoted $\mathrm{F}_t(Y)$, and are obtained by applying the flow complex construction only to critical points $x$ such that $\dd_Y(x)\le t$.
The flow shape $\mathrm{F}_t(Y)$ is homotopy equivalent \cite{buchin2008recursive} to the sublevel set $(\mathbb{R}^n)_Y^t = \bigcup_{i=0}^{m} B(y_i,t)$ of $\dd_Y$, hence also to $\operatorname{\check{C}}_t(Y)$ and $\mathrm{Del}_t (Y)$. 
Considering all scales $t$, by \cref{propo:criticalpass2dist} we know that the filtration of flow shapes $\{\mathrm{F}_t(Y)\}_{t\ge 0}$ realizes the minimum number of cells necessary to have  these homotopy equivalences at each scale.
\end{example}

\section{Examples}

\subsection{Distance from a hypersurface}
\label{subsec:disthyper}

Let $Y$ be a smooth compact hypersurface in $\mathbb{R}^n$, and let $x\in \mathbb{R}^n \setminus Y$ be a nondegenerate critical point of $\dd_Y$. Let us suppose that $x$ is a regular value of the normal exponential map $\varphi \colon NY \to \R^n$, so that we are in the situation of \cref{P:regular-value}, and let $\{y_0, \ldots, y_k\}$ be the set $B(x, \dd_Y(x)) \cap Y$ of closest points to $x$ on $Y$. 

For $\varepsilon >0$ small enough as in \cref{P:regular-value}, let $g_{y_i} \coloneqq \text{dist}_{B(y_i, \varepsilon) \cap Y}$.
Then, in a neighborhood of $x$ the function $\dd_Y$ is a continuous selection of $\mathcal{C}^2$ functions, of the form $\min_{i\in \{0,\ldots ,k\} } g_{y_i}$. Letting 
\begin{equation}\label{eq:vi}
v_i \coloneqq \nabla g_{y_i} (x) = \frac{x-y_i}{\| x-y_i \|} = \frac{x-y_i}{\dd_Y (x)},
\end{equation}
we denote by $(\lambda_i)_{i\in \{0,\ldots ,k\}}$ the unique list of nonnegative numbers (see \cref{D:nondegenerate-crit-pt}) such that $\sum_{i=0}^k \lambda_i =1$ and $\sum_{i=0}^k \lambda_i v_i=0$. 

We describe some properties of the Hessian matrices $H(g_{y_i})$ at $x$, using \cite[Theorem~3.2]{ambrosio1996level} (see also \cite[Section~3]{mantegazza2010notes}). 
For every fixed $i\in\{0,\ldots,k\}$, there exists an orthonormal basis $\{e_{i,1}, \ldots , e_{i,n}\}$ of $\mathbb{R}^n$ such that $N_{y_i}Y = \R \{e_{i,n}\}$ in which the matrix $H(g_{y_i})(x)$ is diagonal. The eigenvalue $\beta_{i,n}$ corresponding to the eigenvector $e_{i,n}$ is zero, while the remaining eigenvalues $\beta_{i,1},\ldots ,\beta_{i,n-1}$ are nonzero since $x$ is nondegenerate, and are given by the formula
\begin{equation}\label{eq:betaij}
\beta_{i,j} = \frac{-\kappa_{i,j}}{1 - \dd_Y(x)\kappa_{i,j}}, \quad j=1,\ldots ,n-1,
\end{equation}
where $\kappa_{i,j}$ are the principal curvatures at $y_i$, i.e.\ the eigenvalues of the shape operator at $y_i \in Y$, see \cite[Chapter 8]{Lee1997}, defined using the \emph{inner} unit normal at $y_i$.

We now study the quadratic index $\quadi(x)$ of $\dd_Y$ at $x$, which is the negative inertia index of the convex combination $\sum_{i=0}^k \lambda_i H(g_{y_i})$ (with coefficients $\lambda_i$ defined as above) restricted to the $(n-k)$-dimensional subspace $V(x)= \operatorname{Span}\{v_0,\ldots ,v_k\}^{\perp}$ of $\R^n$, see \cref{D:nondegenerate-crit-pt}. 
As a consequence of the properties we just discussed, one can state sufficient conditions for determining the quadratic index $\quadi (x)$ based on the simple fact that a convex combination of positive (respectively, negative) semidefinite matrices is positive (resp., negative) semidefinite. 
\begin{corollary}\label{coro:quadi}
In the situation we are considering in \cref{subsec:disthyper}, 
\begin{itemize}
\item [-] if $\beta_{i,j}>0$ for all $i\in \{0,\ldots ,k\}$ and $j\in \{1,\ldots ,n-1\}$, then $\quadi(x)=0$; 
\item [-] if $\beta_{i,j}<0$ for all $i\in \{0,\ldots ,k\}$ and $j\in \{1,\ldots ,n-1\}$, then $\quadi(x)$ takes the maximum value $n-k$.
\end{itemize}
\end{corollary}
If not all the eigenvalues $\beta_{i,j}$ have the same sign, determining $\quadi (x)$ is more complicated. A general expression for $\quadi (x)$ in terms of the eigenvalues $\beta_{i,j}$ is out of reach, since in general the matrices $H(g_{y_i})(x)$ cannot be diagonalized in a common basis. However, this is possible in the case of $Y\subset \R^2$, which we now describe more in detail. 

\subsection{Distance from a curve in the plane}
We now focus on the case $n=2$, with $Y$ a smooth compact curve in $\mathbb{R}^2$.
As in \cref{subsec:disthyper}, let $x \in \mathbb{R}^2 \setminus Y$ be a nondegenerate critical point of $\dd_Y$ which is a regular value of the normal exponential map. By nondegeneracy of $x$, the number of its closest points on $Y$ is $2\le \# (B(x,\dd_Y(x))\cap Y) \le 3$.

If $B(x,\dd_Y(x))\cap Y =\{y_0,y_1,y_2\}$, then the piecewise linear index is $\lini(x)=2$ and the quadratic index is $\quadi(x)=0$ since $V(x)= \operatorname{Span} \{v_0,v_1,v_2\}^\perp =0$. Let us consider the case where $B(x,\dd_Y(x))\cap Y =\{y_0,y_1\}$, which means $\lini(x)=1$. As described in \cref{subsec:disthyper}, for each $i\in \{0,1\}$ the exists an orthonormal basis $\{e_{i,1}, e_{i,2}\}$ of $\mathbb{R}^2$ such that $N_{y_i}Y =\R \{e_{i,2}\}$ in which $H(g_{y_i})(x)$ is diagonal. In this case, 
consider $v_0$ and $v_1$ defined as in \cref{eq:vi}, and observe that they are parallel since $x=\frac{1}{2}(y_0 +y_1)$. We have $N_{y_0}Y =N_{y_1}Y =\operatorname{Span} \{v_0,v_1\}\cong \R$, and $T_{y_0}Y = T_{y_1}Y = V(x) = \operatorname{Span} \{v_0,v_1\}^\perp$. We can diagonalize $H(g_{y_0})(x)$ and $H(g_{y_0})(x)$ in the same basis, 
obtaining the matrices 
\[
D_0 = \begin{pmatrix}
\beta_{0,1} & 0 \\
0 & 0 
\end{pmatrix}, \qquad
D_1 = \begin{pmatrix}
\beta_{1,1} & 0 \\
0 & 0 
\end{pmatrix} .
\]

Since $\frac{1}{2}v_0 +\frac{1}{2}v_1 =0$, the quadratic index $\quadi (x)$ is the negative inertia index of the convex combination $\frac{1}{2}D_0 +\frac{1}{2}D_1$ restricted to $V(x)$, which is determined by the sign of $\beta_{0,1}+\beta_{1,1}$:
\[
\quadi (x) = \begin{cases}
  0  & \text{if } \beta_{0,1}+\beta_{1,1}>0 \\
  1 & \text{if } \beta_{0,1}+\beta_{1,1}<0 .
\end{cases}
\]
Note that this refines \cref{coro:quadi}. The case $\beta_{0,1}+\beta_{1,1}=0$ does not happen, as it would violate the nondegeneracy of $x$.

Via an easy calculation, one can see how the principal curvatures determine the sign of $\beta_{0,1}+\beta_{1,1}$, and hence $\quadi (x)$. To simplify notations, here we drop the index $j$, denoting $\beta_i \coloneqq \beta_{i,1}$ and $\kappa_i \coloneqq \kappa_{i,1}$ for $i\in \{0,1\}$. We also denote $d\coloneqq \dd_Y (x)$. 

First we observe that, by \cite[\textsection 6]{milnor1963morse}, $\kappa_i < \frac{1}{d}$, for $i\in \{0,1\}$, which is equivalent to $1-d \kappa_i >0$. 
As we saw, $\quadi (x)$ is determined by the sign of
\begin{equation}\label{eq:betakappa}
\beta_0 + \beta_1 = \frac{-\kappa_0}{1-d \kappa_0} + \frac{-\kappa_1}{1-d \kappa_1} = \frac{2d \kappa_0 \kappa_1 - \kappa_0 - \kappa_1 }{(1-d \kappa_0)(1-d \kappa_1)}.
\end{equation}
The first equality immediately shows that $\quadi (x)=0$ if $\kappa_0,\kappa_1 <0$, and that $\quadi (x)=1$ if $\kappa_0,\kappa_1 >0$.
Since the denominator of the right-hand side is positive, the sign of $\beta_0 + \beta_1$ is the same as the sign of the numerator 
\[
h(\kappa_0,\kappa_1) \coloneqq 2d \kappa_0 \kappa_1 - \kappa_0 - \kappa_1 = 2d \left[ \left( \kappa_0 -\frac{1}{2d}\right)\left( \kappa_1 -\frac{1}{2d}\right) - \frac{1}{4d^2} \right],
\]
which is easy to study. 

The following statement summarizes the possible values of $\lini (x)$ and $\quadi (x)$. 

\begin{corollary}\label{coro:indicesR2}
Let $Y\subset \R^2$ be a smooth compact curve, and let $x \in \mathbb{R}^2 \setminus Y$ be a nondegenerate critical point of $\dd_Y$ which is a regular value of the normal exponential map $\varphi: NY\to \R^2$. Then, the piecewise linear index $\lini(x)$ and the quadratic index $\quadi(x)$ (see \cref{D:nondegenerate-crit-pt}) can be determined as follows:
\begin{itemize}
    \item [-] if $B(x,\dd_Y(x))\cap Y =\{y_0,y_1,y_2\}$, then $\lini(x)=2$ and $\quadi(x)=0$;
    \item [-] if $B(x,\dd_Y(x))\cap Y =\{y_0,y_1\}$, then $\lini (x)=1$, and $\quadi (x)$ is determined by the principal curvatures $\kappa_i$ of $Y$ at $y_i$, for $i\in \{0,1\}$:
\[
\quadi (x) = \begin{cases}
  0  & \text{if } \left( \kappa_0 -\frac{1}{2d}\right)\left( \kappa_1 -\frac{1}{2d}\right) > \frac{1}{4d^2} \\
  1 & \text{if } \left( \kappa_0 -\frac{1}{2d}\right)\left( \kappa_1 -\frac{1}{2d}\right) < \frac{1}{4d^2} .
\end{cases}
\]
\end{itemize}
\end{corollary}

\subsection{The index of a critical point for the distance from a point}In this section we are going to prove the following result, which is useful when one wants to compute the index of a critical point for the function $\dist$ when $Y=\{y\}\not\subseteq X$ (in this case the piecewise linear index is zero, and we are back to the standard setting of Morse theory).

\begin{proposition}Let $X=Z(p)$ be a smooth hypersurface and $Y=\{y\}\in \R^n$ be a generic point. Let also $x\in X$ be a critical point of $\dist$ (nondegenerate by genericity of $y$). Denote by $Q_x\in \mathrm{Sym}(n-1, \R)$ the matrix representing the second fundamental form of $X\hookrightarrow \R^n$ in the direction of $y-x$ with respect to an orthonormal basis of $T_xX$. Then
$$\quadi(x)=\#\left\{\textnormal{eigenvalues of $Q_x$ that are strictly larger than $\frac{1}{\|x-y\|}$}\right\}.$$
\end{proposition}
\begin{proof}We are going to use the discussion from \cref{remark:restriction}, which we recall in this context. We have a manifold $X$ described by the regular equation $p=0$ and we want to find the critical points of the function $f:=\mathrm{dist}_{\{y\}}|_X$, together with their quadratic index. Notice that in this case $f$ is smooth (because we are assuming $y\notin X$) and the critical points are found using the Lagrange multipliers rule: $x\in X$ is critical if and only if there exists $\lambda\neq 0$ such that
$$\nabla p(x)+\lambda \frac{x-y}{\|x-y\|}=0.$$
Given such $\lambda$ (which is unique), the quadratic index of $f$ at $x$ is given by the negative inertia index of 
$$\bigg(\mathrm{H}(p)(x)+\lambda \mathrm{H}(\mathrm{dist}_{\{y\}})(x)\bigg)\bigg|_{T_xX}.$$
Notice that we have the degree of freedom of choosing $-p$ instead of $p$ as a defining polynomial for $X$ and we choose the sign in such a way that $\nabla p(x)$ has the same orientation as $y-x$. Then $\lambda=\|\nabla p(x)\|$. 

Moreover, the Hessian of $\mathrm{dist}_{\{y\}}$ can be computed easily: denoting by $w:=\tfrac{x-y}{\|x-y\|}$,
$$\mathrm{H}(\mathrm{dist}_{\{y\}})(x)=\frac{1}{\|x-y\|}\Id-w w^\top.$$
In particular, since $w\perp T_xX$, as quadratic forms,
$$\bigg(\mathrm{H}(p)(x)+\lambda \mathrm{H}(\mathrm{dist}_{\{y\}})(x)\bigg)\bigg|_{T_xX}=\bigg(\mathrm{H}(p)(x)+\frac{\|\nabla p(x)\|}{\|x-y\|}\Id \bigg)\bigg|_{T_xX}.$$
Recall now that the second fundamental form $q_x$ of $Z(p)$ at $x$ in the direction of $\nabla p(x)$, as a quadratic form, is related to the Hessian of $p$ as follows (this is proved in \cref{L:II} below):
$$\mathrm{H}(p)(x)=-\|\nabla p(x)\|q_x,$$
and therefore, denoting by $\mathrm{i}^+(A)$ and $\mathrm{i}^-(A)$ the number of positive and negative eigenvalues of a symmetric matrix $A$,
\begin{align}\mathrm{i}^{-}\bigg(\mathrm{H}(p)(x)+\frac{\|\nabla p(x)\|}{\|x-y\|}\Id \bigg)\bigg|_{T_xX}&=\mathrm{i}^{-}\bigg(-\|\nabla p(x)\|q_x+\frac{\|\nabla p(x)\|}{\|x-y\|}\Id \bigg)\bigg|_{T_xX}\\
&=\mathrm{i}^{-}\bigg(-q_x+\frac{1}{\|x-y\|}\Id \bigg)\bigg|_{T_xX}\\
&=\mathrm{i}^{+}\bigg(q_x-\frac{1}{\|x-y\|}\Id \bigg)\bigg|_{T_xX}.
\end{align}
From this last identity the conclusion follows.
\end{proof}

\section{Genericity in the algebraic setting}
\label{S:Genericity-algebraic}

Throughout this section, we consider two real $n$-variable polynomials $p$ and $q$ and two algebraic sets $X \coloneqq Z(p)$ and $Y \coloneqq Z(q)$. We assume $n \geq 1$.

\subsection{A multijet parametric transversality theorem}\label{sec:multijet}

In this section we prove a useful technical tool, \cref{L:submersion}. We first recall the following classic result, which we adapt from \cite[Chapter 3, Theorem 2.7]{Hirsch1976}:

\begin{theorem}[Parametric transversality]
\label{T:parametric-transversality}
Let $M$, $N$, $P$ be smooth manifolds, $W \subseteq N$ a submanifold of $N$, and $F \colon P \times M \to N$ a smooth map. For all $p$ in $P$, we denote by $F_p$ the map $F(p, -) \colon M \to N$.

If $F$ is transverse to $W$, then the set $\{p \in P \mid F(p, -) \pitchfork W\}$ is residual in $P$. In other words, for generic $p \in P$, the map $F_p$ is transverse to $W$.
\end{theorem}

As an immediate consequence we have the following corollary.

\begin{corollary}
\label{C:param-trans-cor}
Let $F \colon P \times M \to N$ be a submersion and $W \subseteq N$ a submanifold. If $\operatorname{codim}_N W = \dim M$, then for generic $p \in P$ we have $\dim(F(p, M) \cap W) = 0$. If $\operatorname{codim}_N W > \dim M$, then for generic $p \in P$ we have $F(p, M) \cap W = \varnothing$.
\end{corollary}

We will be applying these transversality results to \emph{jet spaces}.

\begin{definition}[jets and multijets]
Let $n, m, r \geq 0$ be integers. The space of \emph{jets of order $r$ from $\R^m$ to $\R^n$} is
\[
J^r(\R^m, \R^n) \coloneqq \R^m \oplus \bigoplus_{\ell = 0}^r (\Poly_{\ell, m})^n,
\]
where $\Poly_{\ell, m}$ is the space of homogeneous polynomials of degree $\ell$ in $m$ variables.

Let $f \colon \R^m \to \R^n$ be a smooth function and $x \in \R^m$. The \emph{jet of order $r$ of $f$ at $x$} is
\begin{align}
j^r(f)(x) & \coloneqq \left( \frac{\partial^{|\alpha|} f}{\partial x^\alpha}(x) \right)_{\alpha \in \N^m, |\alpha| \leq r} \in J^r(\R^m, \R^n) \\
& = \left(x, f(x), \ldots, \frac{\partial f}{\partial x_i}(x), \ldots, \frac{\partial^r f}{\partial^r x_i}, \ldots \right).
\end{align}
The identification with polynomial spaces comes from identifying the partial derivative $\frac{\partial^{|\alpha|} f}{\partial x^\alpha}(x)$ with the coefficient in front of the monomial $x^\alpha$.

Let $k \geq 1$ be an integer. The space of \emph{$k$-multijets of order $r$} is just $k$ copies of the space of jets of order $r$:
\[
{}_k J^r(\R^m, \R^n) \coloneqq J^r(\R^m, \R^n)^k.
\]
Similarly, let $f \colon \R^m \to \R^n$ be a smooth function and $\overline{x} \coloneqq x_1, \ldots, x_k$ distinct points in $\R^m$ (we write $\overline{x} \in \R^{km} \setminus \Delta$). The \emph{$k$-multijet of order $r$ of $f$ at $x_1, \ldots, x_k$} is
\[
{}_k j^r(\overline{x}) \coloneqq (j^r(f)(x_i))_{i = 1, \ldots, k}.
\]
\end{definition}

We study jets through parametric jet maps of the form
\[
F \colon \left\{ \begin{array}{ccc}
\Poly_d \times (\R^{nk} \setminus \Delta) & \to & {}_k J^r(\R^n, \R) \\
(p, \overline{x}) & \mapsto & {}_k j^r(p)(\overline x).
\end{array} \right.
\]
The next result is a key technical ingredient for the rest of the paper.
\begin{theorem}
\label{L:submersion}
Let $\Poly_d$ be the space of real polynomials of degree $d$ in $\R^n$. There exists a function $d(n, k, r)$ such that, for all $d \geq d(n, k, r)$, the map $F \colon \Poly_d \times (\R^{nk} \setminus \Delta) \to {}_k J^r(\R^n, \R)$ is a submersion.
\end{theorem}

\begin{proof}
Fix $d$ a polynomial degree, $p \in \Poly_d$ a polynomial, and $\overline{x} \coloneqq (x_1, \ldots, x_k) \in (\R^{nk} \setminus \Delta)$ points in $\R^n$. The differential of $F$ at $(p, x)$ is a linear map
\[
DF_{p, \overline{x}} \colon \left\{ \begin{array}{ccc}
T_{p, \overline{x}} \Poly_d \times \R^{nk} = \Poly_d \times \R^{nk} & \to & T_{F(p, \overline{x})} {}_k J^r(\R^n, \R) = \R^{nk} \oplus \bigoplus_{\ell = 0}^r (\Poly_{\ell, n})^k \\
(\dot{p}, \overline{\dot{x}}) & \mapsto & (\overline{\dot{x}}, \ldots, \dot{p}(x_i) + \nabla p(x_i)^\top \dot{x}_i, \ldots, \nabla \dot{p}(x_i) + \nabla^2 p(x_i)^\top \dot{x}_i, \ldots).
\end{array} \right.
\]
Let $(\overline{\dot{y}}, \dot{U}) = (\overline{\dot{y}}, \overline{\dot{u}}^{(1)}, \ldots, \overline{\dot{u}}^{(r)}) \in \R^{nk} \oplus \bigoplus_{\ell = 1}^r (\Poly_{\ell, n})^k$. For $F$ to be a submersion at $(p, \overline{x})$, it suffices to find $(\dot{p}, \overline{\dot{x}})$ such that $DF_{p, \overline{x}}(\dot{p}, \overline{\dot{x}}) = (\overline{\dot{y}}, \dot{U})$. This is linear in the coefficients of $\dot{p}$ and in monomials in $\overline{\dot{x}}$. We find that $\overline{\dot{x}} = \overline{\dot{y}}$, and then there are $k \sum_{\ell = 0}^r \binom{n}{\ell}$ equations in the coefficients of $\dot{p}$. Since the $x_i$ are pairwise distinct points in $\R^n$, these equations are independent. For there to be a solution in general, it suffices that
\[
\dim \Poly_d = \binom{n + d}{d} \geq k \sum_{\ell = 0}^r \binom{n}{\ell}.
\]
We set $d(n, k, r)$ to be the smallest $d$ satisfying this inequality.
\end{proof}

\begin{example}
\label{Ex:poly-degree}
If $k \leq n + 1$ and $r = 1$, then we can take $d(n, k, r) = 3$. Indeed, for all $n \geq 1$, $\binom{n + 3}{3} \geq (n + 1)^2 \geq k (\binom{n}{0} + \binom{n}{1})$.

If $k \leq n + 1$ and $r = 2$, then we can take $d(n, k, r) = 4$. Indeed, we compute
\begin{align}
\dim \Poly_d - k \sum_{\ell = 0}^r \binom{n}{\ell} & \geq \dim \Poly_d - (n + 1) \sum_{\ell = 0}^r \binom{n}{\ell} \\
& = \binom{n + 4}{4} - (n + 1)(\binom{n}{0} + \binom{n}{1} + \binom{n}{2}) \\
&= \frac{(n + 1)(n + 2)(n + 3)(n + 4)}{24} - (n + 1)(1 + n + \frac{n (n - 1)}{2}) \\
& = \frac{n + 1}{24} ((n + 2)(n + 3)(n + 4) - 12 (n^2 + n + 2) \\
& = \frac{n + 1}{24} (n^3 - 3n^2 + 14n). \label{E:simplified}
\end{align}
We determine the sign of \eqref{E:simplified} by looking at the derivative of the last factor: $3n^2 - 6n + 14$ has discriminant $36 - 4 \cdot 3 \cdot 14 < 0$ and value $14 > 0$ at $n = 0$, so the derivative is always positive. At $n = 1$, we have $n^3 - 3n^2 + 14n = 12$, so for all $n \geq 1$, \eqref{E:simplified} is positive.
\end{example}

\subsection{Finite number of closest points}In this section we prove an analogue of \cite[Proposition 3]{Yomdin1981}, showing that generically the function $m_Y$ is bounded by $n+1$. 

\begin{theorem}[cf.\ {\cite[Proposition 3]{Yomdin1981}}]
\label{T:Yomdin}
For the generic $q \in \Poly_d$ with $d\geq 2$, writing $Y = Z(q)$,  for every $x\in\R^n$ the set $B(x, \dd_Y(x)) \cap Y$ is a nondegenerate simplex with at most $n + 1$ points.\end{theorem}

\begin{proof}
For all $k \in \{0, \ldots, n + 1\}$, define the parameterized jet map
\[
F_k \colon \left\{
\begin{array}{ccc}
\Poly_d \times \R^n \times (\R^{n(k + 1)} \setminus \Delta) & \to & {}_{k + 1} J^0(\R^n, \R^2) \\
(q, x, \overline{y}) & \mapsto & (\overline{y}, \ldots, q(y_i), \ldots, \| x - y_i \|^2, \ldots)
\end{array}
\right.
\]
induced by the function $y \mapsto (q(y), \| x - y \|^2)$. Its differential at a point $(q, x, \overline y)$ is
\[
DF_{k, (q, x, \overline{y})} \colon \left\{ \begin{array}{ccc}
\Poly_d \times \R^n \times \R^{n (k + 1)} & \to & \R^{n (k + 1)} \times \R^{k + 1} \times \R^{k + 1} \\
(\dot{q}, \dot{x}, \overline{\dot{y}}) & \mapsto & (\overline{\dot{y}}, \ldots, \dot{q}(y_i) + \nabla q(y_i)^\top \dot{y}_i, \ldots, 2(x - y_i)^\top (\dot{x} - \dot{y}_i), \ldots).
\end{array} \right.
\]
Let
\[
W_k \coloneqq \left\{
(\overline z, \overline s, \overline t) \in {}_{k + 1} J^0(\R^n, \R^2)
\left|
\begin{array}{l}
\forall i \in \{0, \ldots, k\}, \; s_i = 0, \\
\forall i \in \{1, \ldots, k\}, \; t_i = t_0
\end{array}
\right.
\right\}.
\]
be a subspace of ${}_{k + 1} J^0(\R^n, \R^2)$ and let us prove that $F_k$ is transverse to $W_k$. Let $(q, x, \overline{y}) \in \Poly_d \times \R^n \times (\R^{n(k + 1)} \setminus \Delta)$ such that $F_k(q, x, \overline{y}) \in W_k$, and let $(\overline{\dot{z}}, \overline{\dot{s}}, \overline{\dot{t}}) \in T_{F_k(q, x, \overline{y})} {}_{k + 1} J^0(\R^n, \R^2)$. We are looking for $(\dot{q}, \dot{x}, \overline{\dot{y}}) \in T_{(q, x, \overline{y})} (\Poly_d \times \R^n \times (\R^{n(k + 1)} \setminus \Delta))$ and $(\overline{\dot{w}}, 0, 0) \in T_{F_k(q, x, \overline{y})} W_k$ such that
\begin{align}
(\overline{\dot{z}}, \overline{\dot{s}}, \overline{\dot{t}}) & = DF_{k, (q, x, \overline{y})} (\dot{q}, \dot{x}, \overline{\dot{y}}) + (\overline{\dot{w}}, 0, 0) \\
& = (\overline{\dot{y}} + \overline{\dot{w}}, \ldots, \dot{q}(y_i) + \nabla q(y_i)^\top \dot{y}_i, \ldots, 2(x - y_i)^\top (\dot{x} - \dot{y}_i), \ldots).
\end{align}
Firstly, we observe that $x$ is different from all the $y_i$. Indeed, if not, then we would have $\| x - y_i \|^2 = 0$ for all $i \in \{0, \ldots, k\}$, which contradicts the fact that the $y_i$ are distinct. So we can set $\dot{x} = 0$ and solve the equations $2(x - y_i)^\top (\dot{x} - \dot{y}_i) = \dot{t}_i$ to obtain the $\dot{y}_i$.
Next, since $d \geq 2$, we can find $\dot{q}$ solving the equations $\dot{q}(y_i) + \nabla q(y_i)^\top \dot{y}_i = \dot{s}_i$. Indeed, we have $k + 1$ equations in the coefficients of $\dot{q}$ and $\dim \Poly_d = \binom{n + d}{d} \geq \binom{n + 2}{2} = \frac{(n + 2)(n + 1)}{2} \geq k + 1$.
Finally, we choose $\overline{\dot{w}} = \overline{\dot{z}} - \overline{\dot{y}}$ and conclude that $F_k$ is transverse to $W_k$.

By the parametric transversality theorem (\cref{T:parametric-transversality}), for generic $q$, $F_k(q, -)$ is transverse to $W_k$. Fix such a $q$, let $x \in \R^n$, and take $k + 1$ points $y_0, \ldots, y_k$ from $B(x, \dd_Y(x)) \cap Y$. Then $F_k(q, x, \overline{y}) \in W_k$, so we can write the transversality property:
\[
\textnormal{im}(DF_k(q, -)_{x, \overline{y}}) + T_{F_k(q, x, \overline{y})} W_k = T_{F_k(q, x, \overline{y})} {}_{k + 1} J^0(\R^n, \R^2),
\]
or more explicitly,
\[
\textnormal{im}(DF_k(q, -)_{x, \overline{y}}) + \R^{n (k + 1)} \times 0^{k + 1} \times \R (1, \ldots, 1) = \R^{n (k + 1)} \times \R^{k + 1} \times \R^{k + 1}.
\]
Using the expression of $DF_k$ above, we know that
\[
DF_k(q, -)_{x, \overline{y}} \colon (\dot{x}, \overline{\dot{y}}) \mapsto (\overline{\dot{y}}, \ldots, \nabla q(y_i)^\top \dot{y}_i, \ldots, 2(x - y_i)^\top (\dot{x} - \dot{y}_i), \ldots).
\]
In particular, when the terms $\nabla q(y_i)^\top \dot{y}_i$ equal $0$, then we also have $(x - y_i)^\top \dot{y}_i = 0$ because $x - y_i$ and $\nabla q(y_i)$ are collinear. Thus we need $((x - y_i)^\top \dot{x})_i$ to span a supplementary space to $\R (1, \ldots, 1)$ (by a dimension argument). For this last part to be satisfied, we need the vectors
\[
\{(x - y_i) - (x - y_0)\}_{i = 1, \ldots, k} = \{y_i - y_0\}_{i = 1, \ldots ,k}
\]
to be linearly independent. This is equivalent to $y_0, \ldots, y_k$ forming a nondegenerate simplex in $\R^n$.

Finally, we pick $q$ transverse to every $W_k$ for $k \in \{1, \ldots, n + 1\}$: this is a generic choice. By the above arguments, any distinct set of at most $n + 2$ points in $B(x, \dd_Y(x)) \cap Y$ forms a nondegenerate simplex. We conclude that $B(x, \dd_Y(x)) \cap Y$ has at most $n + 1$ points (any more would lead to a contradiction with nondegeneracy) and that these points form a nondegenerate simplex.
\end{proof}

\subsection{Finite number of critical points}
In this section we study the structure of critical points for the function $\dist$ for generic $X = Z(p)$ and $Y = Z(q)$.

\begin{definition}[$k$-critical points]
\label{D:k-crit-pt}
Let $X$ be a subset of $\R^n$ and $Y$ a closed definable subset of $\R^n$. For all $k \geq 0$, we define the \emph{$k$-critical points of $\dist$} as the critical points $x \in X$ of $\dist$ such that  $B(x, \dd_Y(x)) \cap Y=\{y_0, \ldots, y_k\}$.
We denote by $C_{k, *}(\dist)$ the set of $k$-critical points of $\dist$ that are not in $Y$. 
\end{definition}

Note that, by \cref{T:Yomdin}, every critical point is a $k$-critical point for some $k \in \{0, \ldots, n\}$.

\begin{definition}[Algebraic critical points]
\label{D:algebraic-critical-points}
Let $d \geq 0$, $p, q \in \Poly_d$, and $k \geq 0$. We define the algebraic set
\[
C_k(p, q) \coloneqq \left\{
\begin{array}{rcl}
x & \in & \R^n, \\
y_0, \ldots, y_k & \in & \R^{n(k + 1)}, \\
\lambda_0, \ldots, \lambda_k & \in & \R^{k + 1}, \\
\mu, r & \in & \R
\end{array}
\left|
\begin{array}{l}
x = \mu \nabla p(x) + \sum_{i = 0}^k \lambda_i y_i, \\
\sum_{i = 1}^k \lambda_i = 1, \\
p(x) = 0, \\
\forall i \in \{0, \ldots, k\}, \; q(y_i) = 0, \\
\forall i \in \{0, \ldots, k\}, \; \|x - y_i\|^2 = r^2, \\
\forall i \in \{0, \ldots, k\}, \; \textnormal{rank}(x - y_i, \nabla q(y_i)) \leq 1
\end{array}
\right.
\right\}.
\]
We then define the \emph{algebraic $k$-critical points w.r.t.\ $p$ and $q$} as the set of points $x \in \R^n$ in the projection of $C_k(p, q)$ on the first component.
\end{definition}

\begin{lemma}
\label{L:algebraic-critical-points}
For generic $p\in \Poly_{d_1}$ and $q\in \Poly_{d_2}$, every critical point $x$ of $\dist$ not in $Y$ is an algebraic $k$-critical point for some $k \in \{0, \ldots, n\}$.
\end{lemma}
\begin{proof}
 Let $x$ be an element of {$X \setminus Y$}. By \cref{propo:subgradient1,propo:subgradient2}, we have
\[
\partial_x f = \operatorname{proj}_{T_x X} \left( \operatorname{co} \left\{ \left. \frac{x - y}{\| x - y\|} \; \right| \; y \in B(x, \dd_Y(x)) \cap Y \right\} \right).
\]
Recall that $x$ is a critical point if $0$ belongs to this set. By \cref{T:Yomdin}, for generic $p$ and $q$, the set $B(x, \dd_Y(x)) \cap Y$ is finite with at most $n + 1$ elements. Moreover, for each point $y$ in $B(x, \dd_Y(x)) \cap Y$, the function $y \mapsto \| x - y \|$ reaches a local minimum at $x$, so $x - y$ is parallel to $\nabla q(y)$.

Therefore, if $x$ is a critical point of $\dist$, then there exist $k \in \{0, \ldots, n\}$, $y_0, \ldots, y_k$ in $Y$, $\mu$ in $\R$, and $\lambda_0, \ldots, \lambda_k$ in $[0, 1]$ such that
\begin{equation}
\label{eq:criticalabove}
\begin{gathered}
\sum_{i = 0}^k \lambda_i = 1, \\
\sum_{i = 0}^k \lambda_i (x - y_i) = \mu \nabla p(x), \\
\forall i \in \{1, \ldots, k\}, \; \| x - y_i \| = \| x - y_0 \|, \\
\forall i \in \{0, \ldots, k\}, \; \operatorname{rank}(x - y_i, \nabla q(y_i)) \leq 1,
\end{gathered}
\end{equation}
and so we conclude that $(x, \overline{y}, \overline{\lambda}, \mu, \dd_Y(x)) \in C_k(p, q)$ and $x$ is an algebraic $k$-critical point.
\end{proof}

\begin{theorem}[Finite critical points]
\label{T:finite-crit-pt}
For $d_1, d_2 \geq 3$ and $p \in \Poly_{d_1}$ and $q\in \Poly_{d_2}$, let $X=Z(p)$ and $Y=Z(q)$. Then, for the generic choice of $p$ and $q$, there are only a finite number of critical points of $\dist$. Moreover, for every $k\in \{0, \ldots, n\}$, the number of $k$-critical points of $\dist$ not in $Y$ is bounded by
$$\#C_{k, *}(\dist)\leq c(k,n)d_1^nd_2^{n(k+1)},$$
for some $c(k,n)>0$ depending on $n$ and $k$ only.\end{theorem}
\begin{proof}
Consider the map
\[
F^1 \colon \left\{ \begin{array}{ccc}
\Poly_{d_1} \times \Poly_{d_2} \times \R^n \times (\R^{n(k + 1)} \setminus \Delta) \times \R^L & \to  & J^1(\R^n, \R) \times {}_{k + 1} J^1(\R^n, \R) \times \R^L \\
(p, q, x, \overline{y}, \overline{\lambda}, \mu, r) & \mapsto & (x, p(x), \nabla p(x), \overline{y}, q(\overline{y}), \nabla q(\overline{y}), \overline{\lambda}, \mu, r),
\end{array}
\right.
\]
where $\Delta$ is the fat diagonal and $L = k + 3$.  Consider the following subspace of $J^1(\R^n, \R) \times {}_{k + 1} J^1(\R^n, \R) \times \R^L$:
\[
W^1 \coloneqq \left\{
\begin{array}{rcl}
(x, s, u) & \in & J^1(\R^n, \R), \\
(\overline{y}, \overline{t}, \overline{v}) & \in & {}_{k + 1} J^1(\R^n, \R), \\
\overline{\lambda} & \in & \R^{k + 1}, \\
\mu, r & \in & \R
\end{array}
\left|
\begin{array}{l}
x = \mu u + \sum_{i = 0}^k \lambda_i y_i, \\
\sum_{i = 0}^k \lambda_i = 1, \\
s = 0, \\
\forall i \in \{0, \ldots, k\}, \; t_i = 0, \\
\forall i \in \{0, \ldots, k\}, \; \| x - y_i \|^2 = r^2, \\
\forall i \in \{0, \ldots, k\}, \; \textnormal{rank}(x - y_i, v_i) \leq 1 \\
\end{array}
\right.
\right\}.
\]
We observe that the projection of $\operatorname{im} F^1(p, q, -) \cap W^1$ onto the component $x$ contains the set of algebraic $k$-critical points w.r.t.\ $p$ and $q$. Checking the equations in the definition of $W^1$ from the bottom up, we see that they are independent, so $W^1$ has codimension
\begin{align}
\operatorname{codim} W^1 & = (k + 1) (n - 1) + (k + 1) + (k + 1) + 1 + 1 + n \\
& = nk + 2n + k + 3 \\
& = \dim(\R^n \times (\R^{n(k + 1)} \setminus \Delta) \times \R^L).
\end{align}
By \cref{L:submersion} and \cref{Ex:poly-degree}, for $d_1, d_2 \geq 3$, the map $F^1$ is a submersion. By \cref{C:param-trans-cor}, for generic $p, q$ the set of algebraic $k$-critical points is an algebraic set of dimension $0$, so it is finite. By \cref{L:algebraic-critical-points}, we conclude that the number of critical points is finite.

In order to estimate the cardinality of $C_{k, *}(\dist)$ under the assumptions of the statement, we observe that  this set is the projection on the $x$-coordinate of the set $\Lambda_k \subset \R^n\times \R^{n(k+1)}\times \R^{k+1}\times \R$ consisting of the points $(x, \overline{y}, \overline{\lambda}, \mu)$  solving the system \eqref{eq:criticalabove} with $x\in Z(p)$, $y_i\in Z(q)$ for all $i\in \{0, \ldots, k\}$. Since $C_{k, *}(\dist)$ is finite, its cardinality equals the number of its connected components; moreover each component of $\Lambda_k$ projects to a component of $C_{k, *}(\dist)$ and we have
$$\#C_{k, *}(\dist)=b_0(C_{k, *}(\dist))\leq b_0(\Lambda_k)\leq b(\Lambda_k),$$
where $b(\Lambda_k)$ denotes the sum of the Betti numbers of $\Lambda_k$. In order to control this we use \cite[Theorem 2.6]{basu2018multi} which states that the sum of the Betti numbers of an algebraic set $V\subset \R^{n_1+\cdots +n_a}$ defined by $\ell$ polynomials $p_j(x_1, \ldots, x_a)$ of degree $d_i$ in the variable $x_i\in\R^{n_i}$ can be bounded by
$$b(V)\leq O(1)^{n_1+\cdots+n_a} a^{3(n_1+\cdots+n_a)}d_1^{n_1}\cdots d_p^{n_a}.$$
In our case $a=4$, $(n_1, \ldots, n_4)=(n, n(k+1), k+1, 1)$ and $(d_1, \ldots, d_4)=(\deg(p), \deg(q), 1,1)$, so that
$$b(\Lambda_k)\leq c(k,n)\deg(p)^n\deg(q)^{n(k+1)}.$$
\end{proof}

\begin{remark}
\label{R:bottleneck-bounds}
The same proof works also in the case $(X, Y)=(\R^n, Z(q))$ (see also \cref{R:X-is-Rn-case}), in which case (following the above notation)
$$(d_1, \ldots, d_4)=(1, \deg(q), 1,1).$$ The statement in this case is the following. For the generic polynomial $q\in\Poly_d$ the function $\mathrm{dist}_{Z(q)}$ has only finitely many critical points, all these critical points are algebraic and the number of algebraic $k$-critical points is bounded by
\[
\#C_{k, *}(\mathrm{dist}_{Z(q)})\leq \widetilde{c}(k,n) d^{n(k+1)}.
\]
In particular this recovers and generalizes to all dimensions the bound $\#C_{1, *}\leq c(n)d^{2n}$ that one can obtain by combining \cite[Corollary 4.5]{Eklund2023} with the EDD estimate \cite[Corollary 2.10]{EDD_Draisma}.
\end{remark}

\subsection{Continuous selections}

\begin{proposition}
\label{P:exp-reg-alg}
Let $x \in \R^n$. For $d_1 \geq 3$ and $d_2 \geq 4$ and generic $p\in \Poly_{d_1}$ and $q\in \Poly_{d_2}$, the critical points of $\dist$ are regular values of the normal exponential map of $Y$.
\end{proposition}

\begin{proof}
The normal exponential map of $Y$ is
\[
\exp \colon \left\{ \begin{array}{ccc}
NY \cong Y \times \R & \to & \R^n \\
(y, t) & \mapsto & y + t \nabla q(y).
\end{array} \right.
\]
Its differential at a point $(y, t) \in NY$ is
\[
D \exp_{(y, t)} \colon \left\{ \begin{array}{ccc}
TNY \subseteq \R^n \times \R & \to & T \R^n \cong \R^n \\
(\dot{y}, \dot{t}) & \mapsto & \dot{y} + t H(q)(y) \dot{y} + \dot{t} \nabla q(y).
\end{array} \right.
\]
Thus a point $x$ in $\R^n$ is a critical value of $\exp$ if, and only if, it can be written as $x = y + t \nabla q(y)$ for some $(y, t)$ in $NY$ such that the block matrix $[\Id + t H(q)(y), \nabla q(y)]$ restricted to $\nabla q(y)^\bot \times \R$ has rank strictly less than $n$. Consider the map
\[
F^2 \colon \left\{ \begin{array}{ccc}
\Poly_{d_1} \times \Poly_{d_2} \times \R^n \times (\R^{n(k + 1)} \setminus \Delta) \times \R^L & \to & J^1(\R^n, \R) \times {}_{k + 1} J^2(\R^n, \R) \times \R^L \\
(p, q, x, \overline{y}, \overline{\lambda}, \mu, r) & \mapsto & (x, p(x), \nabla p(x), \overline{y}, q(\overline{y}), \nabla q(\overline{y}), H(q)(\overline{y}), \overline{\lambda}, \mu, r),
\end{array} \right.
\]
and let
\[
W^2 \coloneqq \{(x, s, u, \overline{y}, \overline{t}, \overline{v}, \overline{H}, \overline{\lambda}, \mu, r) \in J^1(\R^n, \R) \times {}_{k + 1} J^2(\R^n, \R) \times \R^{L} \mid (x, s, u, \overline{y}, \overline{t}, \overline{v}, \overline{\lambda}, \mu, r) \in W^1\},
\]
where $W^1$ is defined in the proof of \cref{T:finite-crit-pt}. We define the subspace of $J^1(\R^n, \R) \times {}_{k + 1} J^2(\R^n, \R) \times \R^L$
\[
W^2_{\textnormal{crit}} \coloneqq \left\{
(x, s, u, \overline{y}, \overline{t}, \overline{v}, \overline{H}, \overline{\lambda}, \mu, r) \in W^2
\mid
\forall i \in \{0, \ldots, k\}, \; \textnormal{rank}[(\Id + r H_i) (\Id - v_i v_i^\top), v_i] < n
\right\}.
\]
This space is defined by the same equations as $W^1$ and additional equations that involve the $H_i$. These new equations are independent of those defining $W^2$. Indeed, given $(x, s, u, \overline{y}, \overline{t}, \overline{v}, \overline{H}, \overline{\lambda}, \mu, r)$ in $W^2$ and $i \in \{0, \ldots, k\}$, we can assume $H_i = 0$, which gives us $\operatorname{rank}[(\Id - v_i v_i^\top), v_i] = n$.

We deduce that $W^2_{\textnormal{crit}}$ has codimension $> nk + 2n + k + 3$, which is the dimension of $\R^n \times (\R^{n(k + 1)} \setminus \Delta) \times \R^L$. 
By \cref{L:submersion} and \cref{Ex:poly-degree}, for $d_1 \geq 3$ and $d_2 \geq 4$, the map $F^2$ is a submersion.
By \cref{C:param-trans-cor}, for generic $p$ and $q$, the critical points of $\dist$ miss $W^2_{\textnormal{crit}}$. We conclude that, for generic $p$ and $q$, the critical points of $\dist$ are regular values for the normal exponential map of $Y$.
\end{proof}

\begin{corollary}
\label{C:C2selection}
For generic $p\in \Poly_{d_1}$ and $q\in \Poly_{d_2}$ of respective degrees $d_1 \geq 3$ and $d_2\geq 4$, denoting by $X=Z(p)$ and $Y=Z(q)$ the function $\dist$ is a continuous selection of $\mathcal{C}^2$ functions around its critical points.
\end{corollary}

\begin{proof}
By \cref{P:exp-reg-alg}, for generic $p$ and $q$, the critical points of $\dist$ are regular values of the normal exponential map of $Y$. By \cref{P:regular-value}, we deduce that $\dist$ is a $\mathcal{C}^2$ selection around its critical points.
\end{proof}

\subsection{Nondegeneracy of critical points}

Now we aim to describe the degeneracy of critical points algebraically. To do this, we need to describe the Hessian of $\dist$ at $x$ in terms of the derivatives of $p$ and $q$. 

\begin{lemma}
\label{L:II}
Let $Y = Z(q)$ and $y_0 \in Y$. Then the second fundamental form of $Y$ at $y_0$, is
\[
\textnormal{II}_{y_0} \colon \left\{ \begin{array}{ccc}
T_{y_0} Y \times T_{y_0} Y & \to & N_{y_0} Y = \R \nabla q(y_0) \\
(v, w) & \mapsto & -\frac{\nabla q(y_0)}{\| \nabla q(y_0) \|} v^\top H(q)(y_0) w.
\end{array} \right.
\]
\end{lemma}
\begin{proof}
Let $v_0, w_0 \in T_{y_0} Y$ be tangent vectors and write $\alpha \coloneqq \| \nabla q(y_0) \|$. Define the tangent vector field
\[
v \colon y \mapsto \frac{1}{\alpha^2}\left(v_0 \| \nabla q(y) \|^2 - (v_0^\top \nabla q(y)) \nabla q(y)\right)
\]
in a neighborhood around $y_0$. Note that $v(y) \in T_y Y$ for all $y$ and $v(y_0) = v_0$. By one definition of the second fundamental form \cite[Chapter 8]{Lee1997}, we have 
\[
\text{II}_y(v_0, w_0) = \frac{\nabla q(y_0)}{\alpha} \nabla q(y_0)^\top J(v)(y_0) w_0
\]
(i.e.\ the projection onto $N_y Y$ of the differential of $v$ in the direction $w_0$), where $J(v)$ is the Jacobian of $v$.

Recall that $J(v)$ is the matrix $(\frac{\partial v_i}{\partial x_j})_{1 \leq i, j \leq n}$, where $v_i$ is the $i^{\text{th}}$ component of $v(y)$. For all $i \in \{1, \ldots, k\}$, we have
\[
\alpha^2 v_i(y) = v_{0, i} \| \nabla q(y) \|^2 - (v_0^\top \nabla q(y)) \frac{\partial q}{\partial x_i}(y).
\]
Then, for all $j \in \{1, \ldots, n\}$, we have
\[
\alpha^2 \frac{\partial v_i}{\partial x_j}(y) = v_{0, i} \cdot 2 \nabla q(y)^\top \frac{\partial \nabla q}{\partial x_j}(y) - \left(v_0^\top \frac{\partial \nabla q}{\partial x_j}(y)\right) \frac{\partial q}{\partial x_i}(y) - (v_0^\top \nabla q(y)) \frac{\partial^2 q}{\partial x_i x_j}(y).
\]
Thus we can write
\[
\alpha^2 \nabla v_i (y)^\top = 2 v_{0, i} \nabla q(y)^\top H(q)(y) - (v_0^\top H(q)(y)) \frac{\partial q}{\partial x_i}(y) - (v_0^\top \nabla q(y)) \frac{\partial \nabla q}{\partial x_i}(y)^\top
\]
and
\[
\alpha^2 J(v)(y) = 2 v_0 \nabla q(y)^\top H(q)(y) - \nabla q(y) v_0^\top H(q)(y) - (v_0^\top \nabla q(y)) H(q)(y).
\]
Finally, we get
\begin{align}
\text{II}_{y_0}(v_0, w_0) = {} & \frac{\nabla q(y_0)}{\alpha} \nabla q(y_0)^\top J(v)(y_0) w_0 \\
= {} & \frac{\nabla q(y_0)}{\alpha^3} \Big( 2(\nabla q(y_0)^\top v_0) \nabla q(y_0)^\top - \nabla q(y_0)^\top \nabla q(y_0) v_0^\top - \nabla q(y_0)^\top (v_0^\top \nabla q(y_0)) \Big) H(q)(y_0) w_0 \\
= {} & -\frac{\nabla q(y_0)}{\alpha} v_0^\top H(q)(y_0) w_0,
\end{align}
using the fact that $v_0^\top \nabla q(y_0) = 0$.
\end{proof}

For all $y$ in $Y$, we define $g_y \coloneqq \text{dist}_{B(y, \varepsilon) \cap Y}$ and $f_y \coloneqq g_y |_X$ for $y \in Y$.
We also introduce the following notation: for a symmetric matrix $H\in \mathrm{Sym}(n, \R)$ and a subspace $V\subseteq \R^n$, we denote by 
$H|_V$ a matrix representing the restriction of the quadratic form $h(x)=x^\top Hx$ to $V$. If $L=[v_1, \ldots, v_\ell]\in \R^{n\times \ell}$ is a matrix whose columns form a basis for $V$, then one can take
$$H|_{V}=L^\top HL.$$

\begin{definition}[Algebraic degenerate $k$-critical points]
The set of \emph{algebraic degenerate $k$-critical points w.r.t.\ $p$ and $q$} is the projection of the set
\begin{equation}
\label{E:alg-degenerate}
D_k(p, q) \coloneqq \left\{
(x, \overline{y}, \overline{\lambda}, \mu, r) \in C_k(p, q)
\; \left| \;
\det \left(\mu H(p)(x)|_{\partial_x f^\bot} + r \sum_{i = 0}^k \lambda_i H(g_{y_i})(x)|_{\partial_x f^\bot} \right) = 0
\right.
\right\}
\end{equation}
onto the first component $x$.
\end{definition}

The use of the adjective ``algebraic'' is justified by the following lemma.

\begin{lemma}
\label{L:alg-deg-is-alg}
For all $k \in \{1, \ldots, n\}$, the set $D_k(p, q)$ is algebraic.
\end{lemma}

\begin{proof}
Let $(x, \overline{y}, \overline{\lambda}, r) \in D_k(p, q)$. For all $i \in \{0, \ldots, k\}$, write $\nu_i \coloneqq \frac{\nabla q(y_i)}{\| \nabla q(y_i) \|}$ the unit normal vector from $Y$ at $y_i$. By \cref{L:II}, we have the second fundamental form $\textnormal{II}_{y_i}$ of $Y$ at $y_i$, defined as a bilinear form from $T_{y_i} Y$ to $N_{y_i} Y$. Taking the scalar product with $\nu_i$ gives a quadratic form represented by the matrix $-H(q)(y_i) / \| \nabla q(y_i) \|$. We extend the matrix representation of this form to $N_{y_i} Y$ as follows:
\[
Q_i \coloneqq (\Id - \nu_i \nu_i^\top) \frac{-H(q)(y_i)}{\| \nabla q(y_i) \|} (\Id - \nu_i \nu_i^\top),
\]
with $\Id$ the identity matrix (note that the orientation of $\nu_i$ does not matter). By \cite[Lemma 14.17]{GilbargTrudinger1983}, for each $i \in \{0, \ldots, k\}$, the Hessian of $g_{y_i}$ and $Q_i$ are diagonalizable in the same basis (one consisting of the directions of principal curvature of $Y$ and of a normal vector, see also \cref{subsec:disthyper}), and the eigenvalues of $H(g_{y_i})$ are the image of those of $Q_i$ via the function $\kappa \mapsto \frac{-\kappa}{1 - \| x - y_i \| \kappa}$. We deduce that
\[
H(g_{y_i})(x) = \frac{-Q_i}{1 - \| x - y_i \| Q_i}.
\]
Let $\nu \coloneqq \frac{\nabla p(x)}{\| \nabla p(x) \|}$ and $V \coloneqq [\nu, \nu_0, \ldots, \nu_k] \in \R^{n \times (k + 2)}$ be the matrix of normal vectors. We want to project $H(p)(x)$ and the $H(g_{y_i})(x)$ onto $\partial_x f^\bot \cap \nabla p(x)^\bot$ and then take the convex combination. Since we end up with an $n \times n$ matrix, the equation in \eqref{E:alg-degenerate} becomes
\[
\operatorname{rank} \left( \mu L^\top H(p)(x) L + r \sum_{i = 0}^k \lambda_i L^\top H(g_{y_i})(x) L \right) \leq n - k - 3,
\]
where $L = \Id - V (V^\top V)^{-1} V^\top$, and this equation is algebraic.
\end{proof}

\begin{lemma}
\label{L:deg-is-alg-deg}
For generic $p \in \Poly_{d_1}$ and $q \in \Poly_{d_2}$ with $d_1\geq 3$ and $d_2\geq 4$, every degenerate critical point of $\dist$ not in $Y$ is an algebraic degenerate $k$-critical point for some $k \in \{1, \ldots, n\}$.
\end{lemma}

\begin{proof}
We pick generic $p$ and $q$ such that, for all $x$ in $X$, the set $B(x, \dd_Y(x)) \cap Y$ is a nondegenerate simplex (\cref{T:Yomdin}) and that $\dist$ is a $\mathcal{C}^2$ selection around each of its critical points (\cref{C:C2selection}). In this context, a critical point is degenerate if, and only if, condition (2) of \cref{D:nondegenerate-crit-pt} holds.

Explicitly, let $x$ be a critical point of $\dist$ not in $Y$ and $\{y_0, \ldots, y_k\} = B(x, \dd_Y(x)) \cap Y$. For all $i \in \{0, \ldots, k\}$, the differential of $f_{y_i}$ at $x$ is the projection of $\nabla g_{y_i}(x)$ on $T_x X$: explicitly, there exists $\mu_i \in \R$ such that
\begin{align}
\label{E:finding-mu}
\nabla f_{y_i}(x) & = \nabla g_{y_i}(x) - \mu_i \nabla p(x) \\
& = \frac{x - y_i}{\|x - y_i\|} - \mu_i \nabla p(x).
\end{align}
Since, by definition of critical points, there exist $\lambda_i$ such that $\sum \lambda_i = 1$ and $\sum \lambda_i \nabla f_{y_i}(x) = 0$, we get
\[
x = \sum_{i = 0}^k (\lambda_i y_i + \lambda_i \mu_i \|x - y_i\| \nabla p(x)).
\]
Let $\mu \coloneqq \sum_I \lambda_i \mu_i \|x - y_i\| = \dd_Y(x) \sum_i \lambda_i \mu_i$. We deduce that $(x, \overline{y}, \overline{\lambda}, \mu, \dd_Y(x))$ is an element of $C_k(p, q)$.

We can write $\dist$ as the continuous selection of $\mathcal{C}^2$ functions
\[
\dist = \min_{i = 0, \ldots, k} f_{y_i}.
\]
Therefore, using \cref{remark:restriction}, we see that the critical point $x$ is degenerate if, and only if, the matrix \[
\sum_{i = 0}^k (\lambda_i H(g_{y_i})(x)|_{\partial_x f^\bot} + \lambda_i \mu_i H(p)(x)|_{\partial_x f^\bot})
=
\frac{\mu}{\dd_Y(x)} H(p)(x)|_{\partial_x f^\bot} + \sum_{i = 0}^k \lambda_i H(g_{y_i})(x)|_{\partial_x f^\bot},
\] is singular. So we conclude that if $x$ is a degenerate critical point of $\dist$, then $(x, \overline{y}, \overline{\lambda}, \mu, \dd_Y(x))$ is an element of $D_k(p, q)$, and so $x$ is an algebraic degenerate $k$-critical point for $k \leq n$.
\end{proof}

\begin{remark}The Morse index of a nondegenerate critical point of $\dist$ in the case $X=Z(p)$ and $Y=\{y\}$ with $y \notin X$ can be computed using the Lagrange's multipliers rule, as follows. First, a critical point of $\dist$ is a point $x\in X$ where 
$$\lambda \nabla p(x)+y-x=0$$
for some $\lambda\neq 0$. Now, assuming that $x$ is a nondegenerate critical point, the matrix representing the Hessian of $\dist$ at $x$ in the basis $\{v_1, \ldots, v_{n - 1}\}$ for $T_xX$ can be written as
$$L^\top\left(\lambda H(p)(x)-\Id\right)L,$$
where $L=[v_1, \ldots, v_{n-1}]$.
\end{remark}

\begin{theorem}[Genericity of nondegeneracy]
\label{T:genericity-nondegeneracy}
For generic $p \in \Poly_{d_1}$ and $q\in \Poly_{d_2}$ with $d_1, d_2 \geq 4$, all critical points of $\dist$ not in $Y$ are nondegenerate.
\end{theorem}

\begin{proof}
Consider the map
\[
F^3 \colon \left\{ \begin{array}{ccc}
\Poly_{d_1} \times \Poly_{d_2} \times \R^n \times (\R^{n(k + 1)} \setminus \Delta) \times \R^L & \to & J^2(\R^n, \R) \times {}_{k + 1} J^2(\R^n, \R) \times \R^L \\
(p, q, x, \overline{y}, \overline{\lambda}, \mu, r) & \mapsto & (x, p(x), \nabla p(x), H(p)(x), \overline{y}, q(\overline{y}), \nabla q(\overline{y}), H(q)(\overline{y}), \overline{\lambda}, \mu, r),
\end{array} \right.
\]
and let
\[
W^3 \coloneqq \{(x, s, u, G, \overline{y}, \overline{t}, \overline{v}, \overline{H}, \overline{\lambda}, \mu, r) \in J^2(\R^n, \R) \times {}_{k + 1} J^2(\R^n, \R) \times \R^{L} \mid (x, s, u, \overline{y}, \overline{t}, \overline{v}, \overline{\lambda}, \mu, r) \in W^1\},
\]
where $W^1$ is defined in the proof of \cref{T:finite-crit-pt}.

Let us find a system equations of equations such that, for generic $p$ and $q$, elements of $F^3(\{p\} \times \{q\} \times D_k(p, q))$ satisfy these equations.
We define the subspace of $J^2(\R^n, \R) \times {}_{k + 1} J^2(\R^n, \R) \times \R^L$
\begin{align}
W^3_{\textnormal{deg}} \coloneqq \Big\{ & (x, s, u, G, \overline{y}, \overline{t}, \overline{v}, \overline{H}, \overline{\lambda}, \mu, r) \in W^3 \\
& \left| \begin{array}{l}
\forall i \in \{0, \ldots, k\}, \; \nu_i \coloneqq \frac{v_i}{\| v_i \|} \in \R^n, \\
\forall i \in \{0, \ldots, k\}, \; Q_i \coloneqq (\Id - \nu_i \nu_i^\top) \frac{H_i}{\|v_i\|} (\Id - \nu_i \nu_i^\top) \in \R^{n \times n}, \\
\forall i \in \{0, \ldots, k\}, \; \widetilde{H}_i \coloneqq \frac{-Q_i}{1 - r Q_i} \in \R^{n \times n}, \\
V \coloneqq [\frac{u}{\|u\|}, v_0, \ldots, v_k] \in \R^{n \times (k + 2)}, \\
L \coloneqq \Id - V (V^\top V)^{-1} V^\top \in \R^{n \times n}, \\
\operatorname{rank} \left( \mu L^\top G L + r \sum_{i = 0}^k \lambda_i L^\top \widetilde{H}_i L \right) \leq n - k - 3 
\end{array}
\right\}.
\end{align}
We need to check that the last inequality is independent of the other equations defining $W^3$. It suffices to find an element of $W^3$ such that this rank is $n - k - 2$. Since the equations defining $W^3$ do not involve $G$ nor $\overline{H}$, we can choose these terms freely. Let $(x, s, u, \overline{y}, \overline{t}, \overline{v}, \overline{H}, \overline{\lambda}, \mu, r)$ be an element of $W^3$. Without loss of generality, we assume $\lambda_0 \neq 0$ and we set $G = H_1 = \cdots = H_n = 0$. Up to scaling, we can assume that $r \neq 1$, and up to orthogonal transformation of $H_0$, we can assume $\nu_0$ is $e_n$, the last element of the canonical basis of $\R^n$. In this case, $\| v_0 \| = 1$, so $Q_0 = \operatorname{diag}(1, \ldots, 1, 0)$ and $\widetilde{H}_0 = \frac{-1}{1 - r} Q_0$. Note that $\widetilde{H}_0^\bot = \R e_n = \R \nu_0$, so the rank of $\lambda_0 L^\top \widetilde{H}_0 L$ is maximal; that is, its rank is $n - k - 2$.

We deduce that $W^3_{\textnormal{deg}}$ has codimension $> nk + 2n + k + 3$, which is the dimension of $\R^n \times (\R^{n(k + 1)} \setminus \Delta) \times \R^L$. By \cref{L:submersion} and \cref{Ex:poly-degree}, for $d_1, d_2 \geq 4$, the map $F^3$ is a submersion.
By \cref{C:param-trans-cor}, for generic $p$ and $q$, the image $\operatorname{im} F^3(p, q, -)$ misses $W^3_{\textnormal{deg}}$. In particular, the image of $D_k(p, q)$ by $F^3(p, q, -)$ misses $W^3_{\textnormal{deg}}$. By \cref{L:alg-deg-is-alg,L:deg-is-alg-deg}, we conclude that, for generic $p$ and $q$, the critical points of $\dist$ not in $Y$ are nondegenerate.
\end{proof}

\begin{remark}
\label{R:X-is-Rn-case}
The proofs throughout \cref{S:Genericity-algebraic} do not work as-is for the case $X = \R^n$ because we require $p$ to be of degree at least $3$. However, the results are still true, as can be verified by running through the proofs of \cref{T:finite-crit-pt,P:exp-reg-alg,T:genericity-nondegeneracy} with the parameters $p$, $\mu$, $s$, $u$, and $G$ removed from the definition of $W^1$, $W^2$, and $W^3$, and noting that \cref{L:submersion} can still be applied.
\end{remark}

\bibliographystyle{alpha}
\bibliography{references}

\end{document}